\documentclass[final, 11pt]{amsart}

\usepackage{mathabx}
\usepackage[utf8]{inputenc} 
\usepackage[english]{babel}
\usepackage{amstext}
\usepackage{euscript}
\usepackage{bbm}
\usepackage{mathrsfs}
\usepackage{enumerate}
\usepackage{graphicx}
\usepackage{caption}
\usepackage{amscd}
\usepackage{verbatim}
\usepackage{amssymb}
\usepackage{amsthm}
\usepackage{amsmath}
\usepackage{amstext}
\usepackage{amsfonts}
\usepackage{latexsym}
\usepackage{mathtools}
\usepackage{enumitem} 
\usepackage{accents}
\usepackage{stmaryrd}
\usepackage{float}

\usepackage{textcomp} 
\usepackage{cite} 

\usepackage[left=1.2in, right=1.2in, top=1.2in, bottom=1.2in]{geometry}

\usepackage[usenames,dvipsnames]{xcolor} 
\definecolor{dkblue}{RGB}{1,31,91} 
\definecolor{oxfordblue}{RGB}{4,30,66}

\definecolor{mycolor}{rgb}{0.97,0.98,0.99}
\definecolor{mysecondcolor}{rgb}{0.2,0.2,0.8}

\definecolor{mydarkgreen}{RGB}{42,55,46}
\definecolor{mygreenone}{RGB}{76,131,122}
\definecolor{mytan}{RGB}{225,221,191}
\definecolor{mytant}{RGB}{247.5,243.1,210.1}
\definecolor{mydarkblue}{RGB}{4,37,58}
\definecolor{mydarkbluet}{RGB}{6,55.5,87}
\definecolor{mydarkbluett}{RGB}{12,111,174}

\usepackage[colorlinks=true, pdfstartview=FitV, linkcolor=dkblue, citecolor=dkblue, urlcolor=dkblue]{hyperref}

\usepackage{todonotes}

\usepackage[color]{showkeys}

\definecolor{myred}{rgb}{0.7,0.1,0.1}
\definecolor{mygreen}{rgb}{0.1,0.7,0.1}
\definecolor{myblue}{rgb}{0.2,0.2,0.5}

\theoremstyle{plain}
\newtheorem{thm}{Theorem}
\newtheorem{remark}[thm]{Remark}
\newtheorem{prop}[thm]{Proposition}

\newtheorem{lemma}[thm]{Lemma}
\newtheorem{defn}[thm]{Definition}

\numberwithin{equation}{section}
\numberwithin{thm}{section}

\newcommand{\sgn}{\mathrm{sgn}}
\newcommand{\inum}{\mathrm{i}}

\newcommand{\eqdef}{\overset{\mbox{\tiny{def}}}{=}}

\newcommand{\paren}[1]{\left(#1\right)}
\newcommand{\lclose}[1]{\left[#1\right)}
\newcommand{\jump}[1]{\llbracket#1\rrbracket}
\newcommand{\braces}[1]{\left\{ #1 \right\}}
\newcommand{\set}[2]{\left\{ #1 \left| #2 \right. \right\}}
\newcommand{\D}[2]{\frac{d#1}{d#2}}
\newcommand{\PD}[2]{\frac{\partial#1}{\partial#2}}
\newcommand{\p}{\partial}

\newcommand{\PPD}[2]{{\p_{#1}#2}}
\newcommand{\PPDD}[3]{{\p^{#1}_{#2} #3}}

\newcommand{\at}[2]{\left. #1 \right|_{#2}}

\newcommand{\mc}[1]{\mathcal{#1}}
\newcommand{\wh}[1]{\widehat{#1}}

\newcommand{\bm}[1]{\boldsymbol{#1}}
\newcommand{\abs}[1]{\left\lvert #1 \right\rvert}
\newcommand{\norm}[1]{\left\lVert #1 \right\rVert}

\newcommand{\mbs}{{\mathbb{S}^1}}

\begin{document}

\author[E. Garc\'ia-Ju\'arez]{Eduardo Garc\'ia-Ju\'arez}
\address{Departamento de An\'alisis Matem\'atico, Universidad de Sevilla, C/Tarfia s/n, Campus Reina Mercedes, 41012, Sevilla, Spain. \href{mailto:egarcia12@us.es}{egarcia12@us.es}}

\author[P.-C. Kuo]{Po-Chun Kuo}
\address{Department of Mathematics, University of Pennsylvania, David Rittenhouse Lab., 209 South 33rd St., Philadelphia, PA 19104, USA. 
\href{mailto:kuopo@sas.upenn.edu}{kuopo@sas.upenn.edu}}

\author[Y. Mori]{Yoichiro Mori}
\address{Department of Mathematics, University of Pennsylvania, David Rittenhouse Lab., 209 South 33rd St., Philadelphia, PA 19104, USA. \href{mailto:y1mori@math.upenn.edu}{y1mori@math.upenn.edu}}

\date{\today}

\title[The Immersed inextensible interface problem in 2D Stokes Flow]{The Immersed inextensible interface problem \\ in 2D Stokes Flow}

\begin{abstract}
We study the dynamics of an inextensible, closed interface subject to bending forces and immersed in a two-dimensional and incompressible Stokes fluid. We formulate the problem as a boundary integral equation in terms of the tangent angle and demonstrate the well-posedness in suitable time-weighted spaces of the resulting nonlinear and nonlocal system. The solution is furthermore shown to be smooth for positive times. Numerical computations are performed to initiate the study of the long-time behavior of the interface.

\end{abstract}

\maketitle

\setcounter{tocdepth}{2}
\tableofcontents

\section{Introduction}\label{sec1}

\subsection{Problem Formulation}\label{probform}

We consider the problem of a one-dimensional inextensible interface immersed in a two-dimensional Stokes fluid in $\mathbb{R}^2$.
A closed interface $\Gamma$ encloses a simply connected bounded domain $\Omega\subset\mathbb{R}^2$
filled with a Stokes fluid with viscosity $\mu$. The outside region $\mathbb{R}^2\backslash \Omega$ is filled with 
a Stokes fluid of viscosity $1$. The equations satisfied are:
\begin{align}
\label{uin}
\mu \Delta \bm{u}-\nabla p&=0 \text{ in } \Omega, \\
\label{uout}
\Delta \bm{u}-\nabla p &=0 \text{ in } \mathbb{R}^2\backslash \Omega\\
\label{incomp}
\nabla \cdot \bm{u}&=0 \text{ in } \mathbb{R}^2\backslash \Gamma.
\end{align}
Here $\bm{u}$ is the velocity field and $p$ is the pressure. 

We must specify the interface conditions at $\Gamma$.
Parameterize $\Gamma$ 
by the material or Lagrangian coordinate $\mathbb{S}^1=s \in\mathbb{R}/(2\pi\mathbb{Z})$, and let $\bm{X}(s,t)$
denote the coordinate position of $\Gamma$ at time $t$. 
The parametrization is in the counter-clockwise direction, 
so that the interior region $\Omega$ is on the left hand side of the tangent vector 
$\partial \bm{X}/\partial s$.
The interface is inextensible, and thus, the $s$ is also the arc-length coordinate, so that $\abs{\partial \bm{X}/\partial s}=1$.
For any quantity $w$ defined on 
$\Omega$ and $\mathbb{R}^2\backslash \Omega$, we set:
\begin{equation*}
\jump{w}=\at{w}{\Gamma_{\rm i}}-\at{w}{\Gamma_{\rm i}}
\end{equation*}
where $\at{w}{\Gamma_{\rm i,e}}$ are the trace values of $w$ at $\Gamma$ evaluated 
from the $\Omega$ (interior) and $\mathbb{R}^2\backslash\Omega$ (exterior) sides of $\Gamma$.
Let $\bm{n}$ be the outward pointing unit normal vector on $\Gamma$:
\begin{equation*}
\bm{n}=\mc{R}^{-1}\bm{\tau}, \;\bm{\tau}=\p_s \bm{X}=\PD{\bm{X}}{s},\;
\mc{R}=\begin{pmatrix} 0 & -1 \\ 1 & 0 \end{pmatrix},
\end{equation*}
where $\mc{R}$ is the $\pi/2$ rotation matrix.
The interface conditions are:
\begin{align}
\label{ujump}
\jump{\bm{u}}&=0, \; \PD{\bm{X}}{t}=\p_t\bm{X}=\bm{u}(\bm{X},t),\\
\label{stressjump}
\jump{\Sigma\bm{n}}&=\bm{F},\; \Sigma=\begin{cases}
2\mu\nabla_{\rm S}\bm{u}-pI &\text{ in } \Omega,\\
2\nabla_{\rm S}\bm{u}-pI &\text{ in } \mathbb{R}^2\backslash \Omega,
\end{cases}, \; \nabla_{\rm S}\bm{u}=\frac{1}{2}\paren{\nabla \bm{u}+(\nabla\bm{u})^{\rm T}}.
\end{align}
where $I$ is the $2\times 2$ identity matrix.
The first condition is the no-slip boundary condition and the second is the stress balance condition where $\Sigma$
is the fluid stress and $\bm{F}$
is the force exerted by the interface $\Gamma$. 
We let:
\begin{equation}\label{linearF}
\bm{F}=-\p_s^4 \bm{X}+\p_s(\lambda \bm{\tau})
\end{equation}
Here, the force $\bm{F}$ is a sum of a fourth order bending force and a tension term. The tension $\lambda$ is determined so that the following inextensibility condition is satisfied:
\begin{equation}\label{inext}
\abs{\p_s\bm{X}}=1.
\end{equation}
Note that:
\begin{equation}\label{inext1}
\frac{1}{2}\p_t\abs{\p_s\bm{X}}^2=\p_s\bm{X}\cdot \p_s\p_t\bm{X}=\bm{\tau}\cdot \p_s\bm{u}=0,
\end{equation}
where in the second equaltiy, we used the second relation in \eqref{ujump}. Assuming $\abs{\p_s\bm{X}}=1$ at $t=0$,
the above is equivalent to the inextensibility condition. 
We may also rewrite the force $\bm{F}$ using the curvature $\kappa$:
\begin{equation}\label{linearF1}
\begin{split}
\bm{F}&=\paren{\p_s^2\kappa +\frac{1}{2}\kappa^3}\bm{n}+\p_s(\sigma \bm{\tau}), \; \sigma=\lambda+\frac{3}{2}\kappa^2,\\
&\text{ where }\p_s\bm{\tau}=-\kappa \bm{n}, \; \p_s \bm{n}=\kappa \bm{\tau}.
\end{split}
\end{equation}
We shall henceforth use the above expression for $\bm{F}$, in which $\sigma$ is now the unknown tension enforcing inextensibility. 
In the far field, $|\bm{x}|\to \infty$, we impose the condition that $|\bm{u}|\to 0$ and $p\to 0$.
This completes the specification of our problem.

One crucial feature of the above problem is that we must determine the tension $\sigma$ as part of the problem; this is analogous to the determination of pressure in the Navier-Stokes system. This problem can be formalized as the following {\em tension determination problem}, studied in \cite{KuoLaiMoriRodenberg2023}. 
For a fixed configuration $\bm{X}$, consider \eqref{uin}, \eqref{uout} and \eqref{incomp} with the boundary condition:
\begin{equation}\label{tdf}
\jump{\Sigma \bm{n}}=\wh{\bm{F}}+\p_s(\sigma\bm{\tau}), \; \int_0^{2\pi} \wh{\bm{F}}ds=0.
\end{equation}
where $\wh{\bm{F}}$ is a given force density along the interface. The zero average condition is imposed to avoid logarithmic growth of the velocity field at infinity.
The problem is to determine $\sigma$ so that the inextensbility constraint \eqref{inext1} is satisfied. Note that, for fixed $\bm{X}$, this is a linear problem in $\bm{F}$.
In the context of the full dynamic problem introduced above, this tension determination problem must be solved at each time point given $\wh{\bm{F}}=\paren{\p_s^2\kappa +\frac{1}{2}\kappa^3}\bm{n}$.
It is easily checked that this choice of $\wh{\bm{F}}$ satisfies the zero average condition.

Let us make an observation about the uniqueness of the tension determination problem. 
Let us consider the null space, that is, $\sigma$ when $\wh{\bm{F}}=0$. Take the Stokes equation \eqref{uin} and \eqref{uout}, multiply by $\bm{u}$ and integrate by parts using \eqref{tdf}
with $\wh{\bm{F}}=0$. We have:
\begin{equation}
\int_{\Omega} 2\mu\abs{\nabla_{\rm S}\bm{u}}^2d\bm{x}
+\int_{\mathbb{R}^2\backslash\Omega} 2\abs{\nabla_{\rm S} \bm{u}}^2d\bm{x}=\int_0^{2\pi} \p_s(\sigma \bm{\tau})\cdot\bm{u}ds=-\int_0^{2\pi} \sigma (\bm{\tau}\cdot\p_s\bm{u})ds=0,
\end{equation}
where we integrated by parts in $s$ in the second equality and used \eqref{inext1} in the last inequality. From this, it is immediate that $\nabla_{\rm S}\bm{u}=0$ everywhere and hence 
that the pressure $p$ must be constant in $\Omega$ and $\mathbb{R}^2\backslash \Omega $ respectively. Let the difference in these pressures (interior minus exterior) be equal to $\triangle p$.
The stress jump condition \eqref{tdf} yields:
\begin{equation*}
-\triangle p \bm{n}=\p_s(\sigma \tau)=(\p_s\sigma)\bm{\tau}-\sigma \kappa \bm{n}.
\end{equation*}
Thus, we have $\p_s\sigma=0$ and $\sigma \kappa=\text{const}$. 
We thus have two cases, one in which $\kappa$ is constant, in which case $\Gamma$ is a circle, and in the other case in which $\kappa$ is not a constant and $\Gamma$ is not a circle.
When $\Gamma$ is not a circle, $\sigma=0$. If $\Gamma$ is a circle $\sigma$ is an arbitrary constant. Thus when $\Gamma$ is a circle, the tension determination problem is not uniquely solvable.

We note two important features of the problem. First, the total interfacial length remains constant in time:
\begin{equation*}
\abs{\Gamma}=\int_0^{2\pi} \abs{\p_s\bm{X}}ds=2\pi \text{ since } \abs{\p_s\bm{X}}=1.
\end{equation*}
Furthermore, due to the incompressibility condition, the total area of $\Omega$ does not change with time:
\begin{equation*}
\D{}{t}\abs{\Omega}=0.
\end{equation*}
This implies, by the isoperimetric inequality that:
\begin{equation*}
\abs{\Omega}\leq \frac{\abs{\Gamma}^2}{4\pi}=\pi.
\end{equation*}
If the equality holds in the above, then there is no interesting dynamics; the initial circle will remain stationary.
When equality does not hold, the dynamics is not trivial and the equilibrium shape, if it exists, will not be a circle. 
In particular, it is expected that the tension determination problem will be uniquely solvable at each instant.

The other important feature is that we have the following energy identity. An easy calculation demonstrates that:
\begin{align}\label{Eneqn}
\D{}{t}\int_0^{2\pi} \frac{1}{2}\abs{\p_s^2 \bm{X}}^2ds&=\D{}{t}\int_0^{2\pi} \frac{1}{2}\kappa^2 ds=-\int_{\Omega} 2\mu\abs{\nabla_{\rm S}\bm{u}}^2d\bm{x}
-\int_{\mathbb{R}^2\backslash\Omega} 2\abs{\nabla_{\rm S} \bm{u}}^2d\bm{x}.
\end{align}
We notice that for an inextensible interface its curvature coincides with the second derivative $\kappa(s)=\partial_s^2\bm{X}(s)$, and the energy above corresponds to the so-called Willmore energy \cite{Willmore65}. The expression for $\bm{F}$ can in fact be derived taking the first variation of the Willmore energy under the constraint of local inextensibility (see e.g. \cite[Appendix A]{VeerapaneniGueyffierZorinBiros09}).

\subsection{Boundary Integral, Angle Variable Reduction and Main Results}

We shall henceforth focus on the case $\mu=1$, so that the viscosities of the interior and exterior fluids are the same.
Let us rewrite the above problem using boundary integral equations. We may express the solution to our problem using boundary integrals as follows:
\begin{align}\label{singlelayer}
\bm{u}(\bm{x},t)&=\int_{\mathbb{S}^1} G(\bm{x}-\bm{X}(\eta))\bm{F}(\eta)d\eta, \\
G(\bm{x})&=\frac{1}{4\pi}\paren{-\log \abs{\bm{x}}I+\frac{1}{\abs{\bm{x}}^2}\begin{pmatrix}x^2 & xy \\ xy & y^2\end{pmatrix}}, \; 
\bm{x}=(x,y)^{\rm T}\in \mathbb{R}^2,
\end{align}
where $\bm{G}$ is the Stokeslet, the fundamental solution of the 2D Stokes problem.
We note that $\bm{X}$ and $\bm{F}$ (and other variables) depend on $t$, 
but we will often suppress this dependence to avoid cluttered notation.
It is well-known that the the boundary condition \eqref{ujump} and \eqref{stressjump} are automatically satisfied so long as $\bm{X}$ and $\bm{F}$ are smooth enough. Using the second condition in \eqref{ujump}, we thus have
\begin{equation}\label{Xteqn}
\PD{\bm{X}}{t}(s)=\int_{\mathbb{S}^1} G(\Delta \bm{X})\bm{F}(\eta)d\eta,\; \Delta \bm{X}=\bm{X}(s)-\bm{X}(\eta).
\end{equation}
Since $\abs{\p_s\bm{X}}=1$, we may introduce an angle variable $\theta(s)$ so that:
\begin{equation}\label{Xtheta}
\p_s\bm{X}(s)=\bm{\tau}(s)=\begin{pmatrix} \cos(\theta(s))\\ \sin(\theta(s))\end{pmatrix}.
\end{equation}
We notice that $\theta$ is defined as a continuous function from $\mathbb{S}^1$ to $\mathbb{S}^1$. It will be more convenient to work with a continuous periodic function from $\mathbb{S}^1$ to $\mathbb{R}$,
\begin{equation}\label{varphi_def}
    \varphi(s)=\theta(s)-s.
\end{equation}
Note that:
\begin{equation*}
\p_s \bm{\tau}=-\bm{n}\p_s \theta, \;\; \p_s\bm{n}=\bm{\tau}\p_s \theta, \; \kappa=\p_s\theta.
\end{equation*}
We now rewrite the problem in terms of $\varphi$.  
First, we have:
\begin{equation*}
\p_s \bm{u}=\p_s (\p_t \bm{X})=\p_t \p_s\bm{X}=\p_t \bm{\tau}=-\bm{n}\p_t \varphi.
\end{equation*}
Thus, 
\begin{equation}\label{theta_eq00}
\p_t \varphi=-\bm{n}\cdot \p_s\bm{u}, \; \bm{\tau}\cdot \p_s\bm{u}=0.
\end{equation}
Plugging in the expression for \eqref{Xteqn} into the above, we have:
\begin{align}
\p_t \varphi&=-\bm{n}(s)\cdot\p_s\int_{\mathbb{S}^1} G(\Delta \bm{X})\bm{F}(\eta)d\eta,\label{theta_eqn}\\
0&=\bm{\tau}(s)\cdot\p_s \int_{\mathbb{S}^1} G(\Delta \bm{X})\bm{F}(\eta)d\eta,\label{tension_eqn}
\end{align}
where $\bm{F}$ is given by (note \eqref{linearF1} and $\kappa=\p_s\theta$):
\begin{equation}\label{Fel_theta}
\begin{aligned}
\bm{F}&=\Big(\p_s^3 \theta+\frac{1}{2}(\p_s\theta)^3\Big)\bm{n}+\p_s(\sigma \bm{\tau})\\
&=\Big(\p_s^3 \varphi+\frac{1}{2}(1+\p_s\varphi)^3\Big)\bm{n}+\p_s(\sigma \bm{\tau}).    
\end{aligned}
\end{equation}
Equation \eqref{theta_eqn} is an evolution equation for  $\varphi$,
and equation \eqref{tension_eqn} is the inextensibility constraint that determines the tension $\sigma$. 

Implicit in the above reduction is that we can recover $\bm{X}$ from $\varphi$. Note that:
\begin{equation}
\Delta \bm{X}=\bm{X}(s)-\bm{X}(\eta)=\int_\eta^{s} \bm{\tau}(s')ds'.
\end{equation}
This shows that \eqref{theta_eqn} and \eqref{tension_eqn} can be written solely as equations of $\varphi$. The above only reconstructs $\Delta \bm{X}$; we have only determined $\bm{X}$ up to translation. To complete the recovery of $\bm{X}$ from $\varphi$, consider the center position:
\begin{equation}
\bm{X}_{\rm c}=\frac{1}{2\pi}\int_{\mathbb{S}^1} \bm{X}(s)ds.
\end{equation}
From \eqref{Xteqn}, it is immediate that
\begin{equation}
\PD{\bm{X}_{\rm c}}{t}=\frac{1}{2\pi}\int_{\mathbb{S}^1}\int_{\mathbb{S}^1} G(\Delta \bm{X})\bm{F}(\eta)d\eta ds.
\end{equation}
Using the above, we may recover $\bm{X}_{\rm c}$ and thus reconstruct the whole of $\bm{X}$ once $\varphi$ is known. Henceforth, we may thus focus on solving \eqref{theta_eqn} and \eqref{tension_eqn} where $\varphi$ and $\sigma$ are our unknown functions.

Let us now discuss our strategy to prove well-posedness of \eqref{theta_eqn} and \eqref{tension_eqn}. We first use \eqref{tension_eqn} to solve for $\sigma$. We may write \eqref{tension_eqn} as:
\begin{equation}\label{tdp}
\begin{split}
\mc{L}\sigma&\equiv \bm{\tau}(s)\cdot\p_s
\!\int_{\mathbb{S}^1} \!\!\!G(\Delta \bm{X})\p_\eta(\sigma\bm{\tau})d\eta=
-\bm{\tau}(s)\cdot\p_s\!\int_{\mathbb{S}^1} \!\!\!G(\Delta \bm{X})\wh{\bm{F}}d\eta\equiv F,\\
\wh{\bm{F}}&=\paren{\p_\eta^3 \varphi+\frac{1}{2}(1+\p_\eta\varphi)^3}\bm{n}.
\end{split}
\end{equation}
This is simply the tension determination problem discussed in the previous section, written in boundary integral form. 
Once we solve this problem $\sigma$ will be expressed in terms of $\varphi$, which can then be substituted back into \eqref{theta_eqn}, to produce an equation only in $\varphi$.
We now give a quick sketch of the results on the tension determination problem obtained in \cite{KuoLaiMoriRodenberg2023} and expanded upon in Section \ref{sec3}.

As discussed in Section \ref{probform}, the tension determination problem is expected to have a unique solution if and only if $\bm{X}$ is not a circle.
To verify this, we decompose the operator $\mc{L}$ in \eqref{tdp} as follows.
For a function $\bm{Q}(\eta)$ with values in $\mathbb{R}^2$, let:
\begin{equation*}
\p_s \int_{\mathbb{S}^1}G(\bm{X}(s)-\bm{X}(\eta))\bm{Q}(\eta)d\eta=-\frac{1}{4}(\mc{H}\bm{Q})(s)+(\mc{R}\bm{Q})(s),
\end{equation*}
where $\mc{H}$ is the Hilbert transform,
\begin{equation*}
\begin{aligned}
\mathcal{H}\bm{Q}(s)&=\frac{1}{\pi}\int_{\mathbb{S}^1}\frac{\bm{Q}(\eta)}{2\tan(\frac{s-\eta}{2})}d\eta,
\end{aligned}
\end{equation*}
and the operator $\mc{R}$ depends on 
$\bm{X}$, and hence on $\varphi$, and is defined by
\begin{equation}\label{Roperator}
\begin{aligned}
\mathcal{R}\bm{Q}(s)&=\frac{1}{4\pi}\int_{\mathbb{S}^1}K(s,\eta)\bm{Q}(\eta)d\eta,
\end{aligned}
\end{equation}
where the kernel is given by
\begin{equation*}
    \begin{aligned}
    K(s,\eta)&=-\bigg(\frac{\Delta\bm{X}\cdot\partial_s \bm{X}(s)}{|\Delta\bm{X}|^2}-\frac{1}{2}\cot{\Big(\frac{s-\eta}2\Big)}\bigg)I+\partial_s \frac{\Delta\bm{X}\otimes\Delta\bm{X}}{|\Delta\bm{X}|^2}.
    \end{aligned}
\end{equation*}
Let us rewrite \eqref{tdp} using the above decomposition:
\begin{equation}
\mc{L}\sigma=\bm{\tau}\cdot\paren{\frac{1}{4}\mc{H}+\mc{R}}\p_s(\sigma \bm{\tau})=-\bm{\tau}\cdot\paren{\frac{1}{4}\mc{H}+\mc{F}}\wh{\bm{F}}.
\end{equation}
We further rewrite this as:
\begin{equation}\label{Lsigma}
\begin{split}
\mc{L}\sigma=&\frac{1}{4}\mc{H}\p_s\sigma+\frac{1}{4}\bm{\tau}\cdot[\mc{H},\bm{\tau}]\p_s\sigma-\frac{1}{4}\bm{\tau}\cdot[\mc{H},\bm{n}](\sigma(1+\p_s\varphi))+\mc{R}\p_s\sigma\\
=&-\bm{\tau}\cdot\paren{\frac{1}{4}\mc{H}+\mc{R}}\wh{\bm{F}}
\end{split}
\end{equation}
where  the brackets denote a commutator
\begin{equation}\label{comm}
    [\mc{H},f]g:=\mc{H}\paren{fg}-f\mc{H}g.
\end{equation}
For a function $f$ defined on $\mathbb{S}^1$, we note that
\begin{equation*}
\mc{H}\p_s f=\mc{F}_s^{-1}\abs{\xi}\mc{F}_sf=\abs{\nabla}f
\end{equation*}
where $\mc{F}$ is the Fourier transform (Fourier series) and $\abs{\xi}$ is multiplication by the Fourier variable.
The operator $\mc{H}\p_s=\abs{\nabla}$ thus has the effect of taking one derivative. 
In Section \ref{sec2}, we show that, so long as $\varphi$ is sufficiently smooth, $\mc{R}$ and the commutator \eqref{comm}
increases the regularity of the function on which it acts.
The terms involving $\mc{R}$ and the commutators in \eqref{Lsigma} are thus lower order than $\mc{H}\p_s$.
The operator $\mc{L}$ is the sum of $\mc{H}\p_s$ and a relatively compact perturbation. 
Using Fredholm theory, we see that \eqref{tdp} is solvable if the null space of the operator $\mc{L}$ is trivial.
Given our discussion in Section \ref{probform}, \eqref{tdp} is uniquely solvable if and only if $\bm{X}$ is not a circle. 

Equation \eqref{Lsigma} suggests that $\sigma$ is one degree smoother than $\wh{\bm{F}}$. 
Given that the main term in $\wh{\bm{F}}$ is the third derivative of $\varphi$ (see \eqref{tdf}), we expect $\sigma$ to behave like the second derivative 
of $\varphi$ in terms of its regularity. 
This is indeed the case, which is verified in Section \ref{sec3}.

Let us now consider equation \eqref{theta_eqn}. Write $\bm{F}$ in \eqref{Fel_theta} as:
\begin{equation}\label{NormalTanForce}
\bm{F}=F_n\bm{n}+F_\tau\bm{\tau}, \; F_n(\varphi,\sigma)=\p_s^3\varphi+\frac{1}{2}(1+\p_s\varphi)^3-\sigma (1+\p_s\varphi), \; F_\tau(\sigma)=\p_s\sigma.
\end{equation}
Equations \eqref{theta_eqn} can be written as:
\begin{align}
\p_t \varphi&=\frac{1}{4}\mc{H}F_n(\varphi,\sigma)+\frac{1}{4}\bm{n}\cdot[\mc{H},\bm{n}]F_n(\varphi,\sigma)+\frac{1}{4}\bm{n}\cdot[\mc{H},\bm{\tau}]\p_s\sigma\\
\nonumber&-\bm{n}\cdot\mc{R}\big(F_n(\varphi,\sigma) \bm{n}+\bm{\tau}\p_s\sigma\big)=\frac{1}{4}\mc{H}\p_s^3\varphi+\mc{N}(\varphi,\sigma).
\end{align}
For a function $f$ defined on $\mathbb{S}^1$, we have:
\begin{equation}
\mc{H}\p_s^3 f=-\mc{F}_s^{-1}\abs{\xi}^3\mc{F}_sf=-\abs{\nabla}^3f.
\end{equation}
In Section \ref{nonlin}, we establish that $\mc{N}$ is lower order than $\abs{\nabla}^3$. 
This is essentially a consequence of the fact that $\mc{R}$ and the commutator \eqref{comm} increases the regularity of the function on which it acts, and 
the fact that $\sigma$ has regularity of two derivatives of $\varphi$, as discussed above. Recalling \eqref{varphi_def} and using the Duhamel formula, we have:
\begin{equation}\label{duhamel}
\varphi(t)=\exp\paren{-\frac{t}{4}\abs{\nabla}^3}\varphi_0 +\int_0^t \exp\paren{-\frac{1}{4}\abs{\nabla}^3(t-s)}\mc{N}(\varphi(s),\sigma(\varphi(s)))ds. 
\end{equation}
Here, we have written $\sigma$ as being a function of $\varphi$. Indeed, we may invert \eqref{tdp} to obtain $\sigma=\mc{L}^{-1}F$, where $\mc{L}$ and $F$ only depend on $\varphi$.
In the rest of Section \ref{sec4}, we make use of the above expression to obtain a local in time solution via a fixed point argument and establish parabolic smoothing.

We now state our main result. For a function defined on $\mathbb{S}^1$ with values in $\mathbb{R}$, define the H\"older norms as follows:
\begin{equation}
\begin{split}
\norm{u}_{C^0}&=\sup_{s\in \mathbb{S}^1} \abs{u}, \; \norm{u}_{C^k}=\norm{u}_{C^0}+\sum_{\ell=1}^k \norm{\p_s^\ell u}_{C^0},\; k\in \mathbb{N},\\
\norm{u}_{C^{\gamma}}&=\norm{u}_{C^0}+\sup_{s_1,s_2\in \mathbb{S}^1} \frac{\abs{u(s_1)-u(s_2)}}{\abs{s_1-s_2}^{\gamma}}, \; 0<\gamma<1,\\
\norm{u}_{C^{k,\gamma}}&=\norm{u}_{C^{k-1}}+\norm{\p_s^k u}_{C^\gamma}, \; k\in \mathbb{N},\; 0<\gamma<1.
\end{split}
\end{equation}
The Banach space of continuous functions for which the above norms are finite will be denoted $C^{k}(\mathbb{S}^1;\mathbb{R}), C^{k,\gamma}(\mathbb{S}^1;\mathbb{R})$, etc.
We shall also sometimes denote $C^{k,\gamma}$ as $C^{k+\gamma}$.
We similarly define norms for functions with values in $\mathbb{R}^2$. 
We will also often use the shorthand $C^{k,\gamma}(\mathbb{S}^1)$ or $C^{k,\gamma}$ when no confusion can arise. 
Note that smooth functions are not dense in $C^{k,\gamma}(\mathbb{S}^1), 0<\gamma<1$. We denote the completion of smooth functions in $C^{k,\gamma}(\mathbb{S}^1)$ as $h^{k,\gamma}(\mathbb{S}^1)$.

For a function $u$ taking $a\leq t\leq b$, let
\begin{equation*}
\begin{split}
\norm{u}_{C([a,b];C^{k,\gamma})}&=\sup_{a\leq t\leq b} \norm{u}_{C^{k,\gamma}}, \\
\norm{u}_{C^\ell([a,b];C^{k,\gamma})}&=\norm{u}_{C([a,b];C^{k,\gamma})}+\sum_{j=1}^\ell \norm{\p_s^j u}_{C([a,b];C^{k,\gamma})}.
\end{split}
\end{equation*}
The corresponding function spaces will be denoted $C^\ell([a,b];C^{k,\gamma}(\mathbb{S}^1))$. We also let:

\begin{equation*}
C^\ell((a,b];C^{k,\gamma}(\mathbb{S}^1))=\bigcap_{\epsilon>0}\left\{f:(a,b]\rightarrow C^{k,\gamma}(\mathbb{S}^1)  \left| f|_{[a+\epsilon,b]} \in C^\ell([a+\epsilon,b];C^{k,\gamma}(\mathbb{S}^1))\right.\right\}.
\end{equation*}

To guarantee that $\bm{X}$ defines a non-degenerate simple closed curve we must ensure that the arc-chord condition is satisfied,
\begin{equation*}
    \begin{aligned}
        |\bm{X}|_*=\inf_{s_1,s_2\in\mathbb{S}^1,s_1\neq s_2}\frac{|\bm{X}(s_1)-\bm{X}(s_2)|}{|s_1-s_2|}>0.
    \end{aligned}
\end{equation*}
Recalling \eqref{Xtheta}, we will also denote the arc-chord quantity by $|\varphi|_\star$, 
\begin{equation*}
    \begin{aligned}
        |\varphi|_\star=\inf_{s_1,s_2\in\mathbb{S}^1,s_1\neq s_2}\frac{1}{|s_1-s_2|}\Big|\int_{s_1}^{s_2} \begin{pmatrix} \cos(s+\varphi(s))\\ \sin(s+\varphi(s))\end{pmatrix}ds \Big|.
    \end{aligned}
\end{equation*}
We define our solution as follows.
\begin{defn}\label{def_sol}
    Let  $\gamma\in(0,1)$ and $\varphi_0\in C^{1,\gamma}(\mathbb{S}^1)$ with $|\varphi_0|_*>0$. We say that $(\varphi,\sigma)$, $\varphi\in C^1((0,T];C^{1,\gamma}(\mathbb{S}^1))\cap C((0,T];C^{4,\gamma}(\mathbb{S}^1))$, $\sigma\in C((0,T];C^{2,\gamma}(\mathbb{S}^1))$  is a strong solution  if it satisfies the equations \eqref{theta_eqn}-\eqref{tension_eqn} pointwise, $|\varphi(t)|_*>0$ for $t\in[0,T]$, and $\varphi(t)\rightarrow \varphi_0$ in $C^{1,\gamma}(\mathbb{S}^1)$ as $t\to 0^+$.
\end{defn}
Our main result is as follows.
\begin{thm}\label{main_theo}
     Let $\varphi\in h^{1,\gamma}(\mathbb{S}^1)$, $\gamma\in(0,1)$, $|\varphi|_*>0$. Then, there exists $T>0$ such that the Inextensible Interface Problem \eqref{theta_eqn}-\eqref{tension_eqn} has a unique strong solution.
     Moreover, the solution is smooth at positive times $\varphi\in C^\infty((0,T]\times\mathbb{S}^1)$.
\end{thm}
We note in particular that smoothness of the solution for positive time ensures that the solution obtained
 is also a solution to the original PDE formulation of the problem discussed in Section \ref{probform}. 

Let us discuss our choice of function space. 
The existence and uniqueness part of Theorem \ref{main_theo} will be done via a fixed point argument on the mild solution formula \eqref{duhamel} (see Theorem \ref{theo_fixed}).
Note that $\mc{N}(\varphi,\sigma)$ in \eqref{duhamel} contains the third derivative of $\varphi$, and thus, a straightforward application of semigroup theory 
will only yield solutions with initial data in $\varphi_0\in C^{3,\gamma}, \gamma>0$. In order to expand the space of admissible initial data, 
we make use of the time-weighted space
\begin{equation}\label{space_fixed}
    \begin{aligned}
        X(T)&=C([0,T];C^{1,\gamma}(\mathbb{S}^1))\cap C( (0,T]; C^{4,\gamma}(\mathbb{S}^1)) \cap L^\infty (0,T; t C^{4,\gamma}(\mathbb{S}^1)),
    \end{aligned}
\end{equation}
with norm
\begin{equation*}
    \begin{aligned}
        \|f\|_{X(T)}&:=\|f\|_{C([0,T] ;C^{1,\gamma})}+\|t\|f(t)\|_{C^{4,\gamma}}\|_{L^\infty([0,T])}\\
        &= \sup_{t\in[0,T]}\|f(t)\|_{C^{1,\gamma}}+\sup_{t\in[0,T]}t\|f(t)\|_{C^{4,\gamma}}.
    \end{aligned}
\end{equation*}
We note that the restriction $\varphi_0\in C^{1,\gamma}$ comes from the regularity needed to solve the tension determination problem.

As part of the fixed point scheme, the tension equation must be solved at each fixed time $t>0$.
We define 
\begin{equation*}
    \begin{aligned}
        d(\varphi)=\min_{a\in\mathbb{S}^1}\|\varphi(s)-a\|_{C^0}.
    \end{aligned}
\end{equation*}
Since we assume that the length of the interface described by $\varphi$ is $2\pi$, the interface is not a circle if and only if
\begin{equation*}
    \begin{aligned}
        d(\varphi)>0.
    \end{aligned}
\end{equation*}
Because the dynamics for a circular interface is trivial, we can assume that $d(\varphi_0)>0$.
The proof of Theorem \ref{main_theo} shows that the solution time $T$ can be taken to depend only on the smoothness of $\varphi_0$, $d(\varphi_0)$ and $\abs{\varphi_0}_*$
(see Remark \ref{length_of_T}).

A parabolic bootstrapping scheme (see Proposition \ref{smoothing}) allows to show that the parabolic character of the equation regularizes the solution for positive times.

In Section \ref{sec5}, we study numerically the asymptotic behavior of the dynamics.
Let us first discuss the steady states.
Note that, at steady state, by \eqref{Eneqn}, the energy dissipation must be equal to $0$, and hence, $\nabla_{\rm S}\bm{u}=0$. 
Since $\bm{u}\to 0$ as $\bm{x}\to \infty$, this implies that $\bm{u}=0$ in $\mathbb{R}^3\backslash \Omega$. By \eqref{ujump}, 
we see that $\bm{u}=0$ in $\Omega$ as well.
From \eqref{uin} and \eqref{uout} we thus see that the pressure $p$ must be spatially constant within $\Omega$
and $\mathbb{R}^3\backslash \Omega$ respectively. Let $p_{\rm in}$ and $p_{\rm out}$ be these constant pressures 
respectively in $\Omega$ and $\mathbb{R}^3\backslash \Omega$.
Condition \eqref{stressjump} yields:
\begin{equation*}
-\Delta p\bm{n}=\paren{\p_s^3 \varphi+\frac{1}{2}(1+\p_s\varphi)^2-\sigma (1+\p_s\varphi)}\bm{n}+(\p_s\sigma)\bm{\tau},\; \Delta p=p_{\rm in}-p_{\rm out}.
\end{equation*}
From the above, we see that:
\begin{equation*}
\sigma=\sigma_\star=\text{const.}
\end{equation*}
The equation satisfied by $\kappa=1+\p_s\varphi$ is:
\begin{equation*}
\p_s^2 \kappa+\frac{1}{2}\kappa^3=\sigma_\star \kappa-\Delta p.
\end{equation*} 
This same equation can be obtained as the Euler-Lagrange equation for minimizing the bending energy (the curvature squared integral, 
see \eqref{Eneqn}) under the constraint of fixed area and length of the interface. In this context, the constants 
$\Delta p$ and $\sigma_\star$ correspond to the Lagrange multipliers associated with the fixed area and length constraints 
respectively. Veerapaneni, Raj, Biros and Purohit  study this problem in \cite{VeerapaneniRajBirosPurohit09}, where they obtained multiple $m$-fold symmetric steady states.

In Section \ref{sec5}, we develop a numerical method based on our analytical formulation, and perform numerical simulations on the system.
The observed behavior indicates that only the two-lobe steady state is stable, while the other $m$-lobe symmetric states, 
$m\geq 3$ correspond to saddle points.

\subsection{Historical Considerations and Related Problems}
The problem addressed in this paper is a classic model in biophysics. Indeed, in order to explain the biconcave shape of red blood cells, Canham \cite{Canham70} introduced a bending energy dependent on the squared mean curvature. Soon later, to model the closed lipid bilayer in cells, Helfrich 
formulated the renowned Canham-Helfrich energy \cite{Helfrich73, DeulingHelfrich76} with density given by 
$$\frac12(\kappa-\kappa_0)^2+cK,$$
where $\kappa$ denotes the mean curvature, $\kappa_0$ the so-called spontaneous curvature, $c$ is a constant and $K$ the Gaussian curvature (which by the Gauss-Bonnet theorem adds no effect for genus zero surfaces).
The case $\kappa_0=0$, commonly chosen when asymmetries are neglected, is called the Willmore energy \cite{Willmore65}. We consider here the classic model for two-dimensional vesicles such as red blood cells: a closed interface with prescribed length and enclosed area (i.e., with given isoperimetric ratio) whose bending force is dictated by the Willmore energy.

The steady version of the problem has given rise to interesting questions in the calculus of variations and geometric measure theory. In fact, finding the equilibrium shape of the vesicles should amount to minimizing the Willmore, or the Canham-Helfrich, functional with given isoperimetric ratio. See \cite{Schygula-12, MondinoScharrer-20} and the references therein for the existence and regularity of such minimizers in two and three dimensions, and also \cite{NagasawaTakagi03} for the stability of some minimizers bifurcating from the sphere.
Particular analytic steady-state shapes for two-dimensional vesicles (i.e., our interface) are found in \cite{VeerapaneniRajBirosPurohit09} (see also \cite{MarplePurohitVeerapaneni15}) in terms of elliptic integrals.

On the contrary, little rigorous analysis has been performed on the dynamical version of the problem. 
Other than the works \cite {Lengeler-18, KohneLengeler-18}, where the three-dimensional problem is studied for initial data close enough to a smooth surface via the Hanzawa transform,
most results are of computational nature. Indeed, following the immersed boundary method pioneered by Peskin in the study of heart valves \cite{Peskin77, Peskin02}, computational works dealing with this model and closely related ones are avialable \cite{VeerapaneniGueyffierZorinBiros09, RahimianVeerapaneniBiros10}, also in three dimensions  \cite{SeolHuKimLai16, LaiSeol22}. On the analytical side, we point the recent well-posedness results on the related \textit{Peskin problem} \cite{LinTong19, MoriRodenbergSpirn19}, which models the dynamics of an immersed elastic (extensible) membrane, and the  consequent improvements in terms of initial regularity and extended features \cite{CameronStrain23, ChenNguyen23, GancedoGraneroScrobogna21, GJMoriStrain23, GJKuoMoriStrain23, Li21, Tong21, Tong24, TongWei23, GJHaziot2023}. 

We also point out that the inextensibility constraint features in many other fluid structure interaction models. One such model is the swimming filament model in which an inextensibility constraint is imposed on the deforming interface \cite{fauci1988computational,teran2010viscoelastic,thomases2014mechanisms}. We refer to \cite{mori2023well} for an analytical study on a simpler version of this problem. 
    
\section{Calculus Estimates}\label{sec2}

We summarize in this section some calculus estimates that will be repeatedly used along the proof of local well-posedness in the following section.

\begin{lemma}[Interpolation]\label{InterpolationLem}
\begin{enumerate}
    \item 
    Let $k_1,k_2\in\mathbb{Z}$ with $0\leq k_1< k_2$. Then, for  any $k\in \mathbb{Z}$ between $k_1$ and $k_2$,
\begin{align*}
    \norm{f}_{C^k}\leq C \norm{f}_{C^{k_1}}^\frac{k_2-k}{k_2-k_1}\norm{f}_{C^{k_2}}^\frac{k-k_1}{k_2-k_1},
\end{align*}
where $C$ depends on $k_1, k_2$ and $k$.
    \item 
    Let $0<\alpha_1<\alpha_2$ and denote $\alpha=s\alpha_1+(1-s)\alpha_2$ with $s\in(0,1)$. Then for any $\beta \notin \mathbb{Z}$ with $\alpha_1\leq \beta\leq\alpha$, or any $\beta \in \mathbb{Z}$ with $\alpha_1\leq \beta<\alpha$
    \begin{equation*}
    \begin{aligned}
        \|f\|_{C^{\beta}}&\leq C \|f\|_{C^{\alpha_1}}^s\|f\|_{C^{\alpha_2}}^{1-s},
    \end{aligned}
    \end{equation*}
    where $C$ depends on $\alpha_1,\alpha_2$ and $\alpha$.
\end{enumerate}
Moreover, given $f\in X\paren{T}$, with $X(T)$ defined in \eqref{space_fixed} with $\gamma\in(0,1)$, if  $\alpha\in (0,3)$ and $\gamma+\alpha$ is not an integer, we have for all $t\in(0,T]$,
\begin{align}
    \norm{f}_{C^{1+\gamma+\alpha}}\leq C t^{-\frac{\alpha}{3}}\norm{f}_{X\paren{T}}, \quad \alpha\in [0, 3], \label{interplationbddeq03}
\end{align}
where $C$ depends on $\gamma$ and $\alpha$.
If  $\gamma+\alpha$ is an integer,  for any $\delta\in (0, 3-\alpha)$ and $t\in(0,T]$,
\begin{align}
    \norm{f}_{C^{1+\gamma+\alpha}}\leq C t^{-\frac{\alpha+\delta}{3}}\norm{f}_{X\paren{T}},\label{interplationbddeq04}
\end{align}
where $C$ depends on $\gamma, \alpha$ and $\delta$.
\end{lemma}
Next, we show commutator estimates for the Hilbert transform,
\begin{equation*}
    [\mc{H},f]g(s)=\frac{1}{\pi}\int_{\mathbb{S}}\frac{f(\eta)-f(s)}{2\tan{(\frac{s-\eta}{2})}}g(\eta)d\eta,
\end{equation*}
and for the related operator $\mc{R}$ defined in \eqref{Roperator}. 
\begin{lemma}[Commutator]\label{CommutatorLem} 
 Let $0\leq k\in\mathbb{Z}$, $\alpha\in(0,1)$. It holds that:  
\begin{enumerate}
    \item If $f\in C^{k+1}(\mathbb{S}^1)$, $g\in C^k(\mathbb{S}^1)$, then
\begin{equation}\label{comm1}
    \|[\mathcal{H},f]g\|_{C^{k,\alpha}}\leq C \big(\|f\|_{C^{k+1}}\|g\|_{C^0}+\|f\|_{C^{1}}\|g\|_{C^k}\big). 
\end{equation}
\item If $f\in C^{k+2}(\mathbb{S}^1)$, $g\in C^{k}(\mathbb{S}^1)$,
\begin{equation}\label{comm2}
    \|[\mathcal{H},f]{g}\|_{C^{1+k,\alpha}}\leq C \big(\|f\|_{C^{2+k}}\|g\|_{C^{0}}+\|f\|_{C^{2}}\|g\|_{C^{k}}\big).
\end{equation}

\end{enumerate}
\end{lemma}
\begin{proof}
The base case $k=0$ for $(1)$ is proved in \cite{KuoLaiMoriRodenberg2023}:
\begin{equation}\label{comm0}
    \|[\mathcal{H},f]g\|_{C^{\alpha}}\leq C \|f\|_{C^{1}}\|g\|_{C^0} \quad \forall  \alpha\in(0,1).
\end{equation}
    Through Leibniz product rule and \eqref{comm0}, we have 
    \begin{equation*}
        \norm{\PPDD{k}{s}{\paren{[\mathcal{H},f]g}}}_{C^{0,\alpha}}
        =\norm{\sum_{i=0}^k[\mathcal{H},\PPDD{i}{s}{f}]\PPDD{k-i}{s}{g}}_{C^{0,\alpha}}
        \leq C\sum_{i=0}^k \norm{\PPDD{i}{s}{f}}_{C^{1}}\norm{\PPDD{k-i}{s}{g}}_{C^{0}}.
    \end{equation*}
    For each $i$, by the Sobolev interpolation inequality and Young's inequality, we obtain
    \begin{equation*}
    \begin{aligned}
        \norm{\PPDD{i}{s}{f}}_{C^{1}}\norm{\PPDD{k-i}{s}{g}}_{C^{0}}
        \leq C \norm{{f}}_{C^{1}}^{\frac{k-i}{k}}\norm{{f}}_{C^{k+1}}^\frac{i}{k}\norm{{g}}_{C^{0}}^\frac{i}{k}\norm{{g}}_{C^{k}}^\frac{k-i}{k}\\
        \leq C \paren{\norm{{f}}_{C^{k+1}}\norm{{g}}_{C^{0}}+\norm{{f}}_{C^{1}}\norm{{g}}_{C^{k}}},
    \end{aligned}  
    \end{equation*}
    where $C$ depends on $i$ and $k$.
    Therefore, we obtain \eqref{comm1}.
   
   To prove the second point, we first split $[\mathcal{H},f]g$ into
    \begin{equation*}
    \begin{aligned}
        [\mathcal{H},f]g(s)= \frac{1}{\pi} R_c\ast \paren{fg}- \frac{1}{\pi}f \paren{R_c\ast g}-\frac{1}{2\pi}\int_{\mathbb{S}}K_0\paren{s,\eta}g(\eta)d\eta,
    \end{aligned}
    \end{equation*}
    where $\ast$ is the convolution operator, and
    \begin{equation*}
    \begin{aligned}
        R_c(s)=&\frac{1}{2}\cot{\frac{s}{2}}-\frac{1}{s},\\
        K_i\paren{s,\eta}=&\frac{\PPDD{i}{\eta}{f}(\eta)-\PPDD{i}{s}{f}(s)}{\eta-s}.
    \end{aligned}
    \end{equation*}
    By the Taylor expansion of $\cot(s)$, $R_c(s)$ is an analytic function on $\mbs$, so
    \begin{equation*}
        \norm{[R_c\ast,f]g}_{C^{1,\alpha}}\leq C\|f\|_{C^{2}}\|g\|_{C^0},
    \end{equation*}
    where $C$ depends on the expansion of $\cot(s)$.
    For the second part, we use that
    \begin{equation*}
        \paren{\PPD{s}{}+\PPD{\eta}{}}K_0\paren{s,\eta}=K_1\paren{s,\eta},
    \end{equation*}
    \begin{equation*}
        \PPD{s}{}\int_{\mathbb{S}}K_0\paren{s,\eta}g(\eta)d\eta=\int_{\mathbb{S}}K_1\paren{s,\eta}g(\eta)d\eta-\int_{\mathbb{S}}\PPD{\eta}{} K_0\paren{s,\eta}g(\eta)d\eta.
    \end{equation*}
    Through the proof of \cite[Propostion 2.8]{KuoLaiMoriRodenberg2023}, for all $0<\alpha<1$,
    \begin{equation*}
        \norm{\int_{\mathbb{S}}K_1\paren{s,\eta}g(\eta)d\eta}_{C^{1,\alpha}}\leq C\|\PPD{s}{f}\|_{C^{1}}\|g\|_{C^0}\leq C\|f\|_{C^{2}}\|g\|_{C^0}.
    \end{equation*}
    Next, since
    \begin{equation*}
    \begin{aligned}
        \PPD{s}{}\PPD{\eta}{} K_0\paren{s,\eta}
        &=\frac{1}{\paren{\eta-s}^2}\paren{\PPD{s}{} f\paren{s}-\frac{f\paren{\eta}-f\paren{s}}{\eta-s}}\\
        &\quad+\frac{1}{\paren{\eta-s}^2}\paren{\PPD{\eta}{} f\paren{\eta}-\frac{f\paren{\eta}-f\paren{s}}{\eta-s}},
    \end{aligned}
    \end{equation*}
    we have
    \begin{equation*}
        \abs{\PPD{s}{}\PPD{\eta}{} K_0\paren{s,\xi}}\leq C \|f\|_{C^{2}} \frac{1}{\abs{s-\eta}}.
    \end{equation*}
    Therefore, for all $s,\eta\in \mbs$,
    \begin{equation*}
        \abs{\Delta \int_{\mathbb{S}}\PPD{\xi}{} K_0\paren{s,\eta}g(\xi)d\xi}\leq C\|f\|_{C^{2}}\|g\|_{C^0} \abs{s-\eta}\log\abs{s-\eta},
    \end{equation*}
    and we finally obtain
    \begin{equation*}
        \norm{[\mc{H},f]g}_{C^{1,\alpha}}\leq C\|f\|_{C^{2}}\|g\|_{C^0}.
    \end{equation*}
    Then, the case $k\geq1$ follows from the $k=0$ one by taking derivatives and interpolation.

\end{proof}

\begin{prop}[Remainder Estimates]\label{RestimateProp} 
Consider the operator $\mathcal{R}$ defined in \eqref{Roperator} and assume that  $|\bm{X}|_*>0$. Let $0\leq k\in\mathbb{Z}$, $\alpha\in(0,1)$. It holds that:  
\begin{enumerate}
    \item If $\bm{X}\in C^{2+k}(\mathbb{S}^1)$, then $\mathcal{R}$  maps $C^{k}(\mathbb{S}^1)$ into $C^{k,\alpha}(\mathbb{S}^1)$ and
\begin{equation}\label{Rest1}
        \|\mathcal{R}\bm{u}\|_{C^{k,\alpha}}
        \leq C \frac{\|\bm{X}\|_{C^2}}{|\bm{X}|_*^2}\paren{\|\bm{X}\|_{C^2}\|\bm{u}\|_{C^{k}}+\|\bm{X}\|_{C^{2+k}}\|\bm{u}\|_{C^{0}}}.
\end{equation}
\item If $\bm{X}\in C^{1+k,\alpha}(\mathbb{S}^1)$,  then, for $\alpha\neq1/2$, 
\begin{equation}\label{Rest2}
        \|\mathcal{R}\bm{u}\|_{C^{k+\lfloor2\alpha\rfloor,2\alpha-\lfloor2\alpha\rfloor}}
        \leq C \frac{\|\bm{X}\|_{C^1}^2}{|\bm{X}|_*^3}\big(\|\bm{X}\|_{C^{1+k,\alpha}}\|\bm{u}\|_{C^{0,\alpha}}+\|\bm{X}\|_{C^{1,\alpha}}\|\bm{u}\|_{C^{k,\alpha}}\big),
\end{equation}
and if $\alpha=1/2$, then for any $\beta\in(0,1)$,
\begin{equation*}
        \|\mathcal{R}\bm{u}\|_{C^{k,\beta}}
        \leq C \frac{\|\bm{X}\|_{C^1}^2}{|\bm{X}|_*^3}\big(\|\bm{X}\|_{C^{1+k,\alpha}}\|\bm{u}\|_{C^{0,\alpha}}+\|\bm{X}\|_{C^{1,\alpha}}\|\bm{u}\|_{C^{k,\alpha}}\big).
\end{equation*}
\end{enumerate}
\end{prop}

\begin{proof}
   We adapt the proof from \cite[Section 4.2] {Rodenberg:peskin-thesis} and \cite[Proposition 2.6]{KuoLaiMoriRodenberg2023}. 
   The case $k=0$ of the estimate \eqref{Rest1} is contained in \cite{KuoLaiMoriRodenberg2023}, while for $k=1$ we have
    \begin{align*}
    \begin{split}
        \p_s \mathcal{R}\bm{u}(s)
        =&\frac{1}{4\pi}\int_{\mathbb{S}^1}\p_s K(s,\eta)\bm{u}(\eta)d\eta\\
        =&-\frac{1}{4\pi}\int_{\mathbb{S}^1}\p_\eta K(s,\eta)\bm{u}(\eta)d\eta+\frac{1}{4\pi}\int_{\mathbb{S}^1}\paren{\p_s+\p_\eta} K(s,\eta)\bm{u}(\eta)d\eta\\
        =&\frac{1}{4\pi}\int_{\mathbb{S}^1} K(s,\eta)\p_\eta\bm{u}(\eta)d\eta+\frac{1}{4\pi}\int_{\mathbb{S}^1}\paren{\p_s+\p_\eta} K(s,\eta)\bm{u}(\eta)d\eta\\
        :=&I_1+I_2.
    \end{split}   
    \end{align*}
    It follows that 
    \begin{align*}
        \|I_1\|_{C^{0,\beta}}\leq C \frac{\|\partial_s\bm{X}\|_{C^1}^2}{|\bm{X}|_*^2}\|\p_s\bm{u}\|_{C^{0}}.
    \end{align*}
    Next, since $\paren{\p_s+\p_\eta} \Delta f= \Delta \paren{\p_s f}$ and $\paren{\p_s+\p_\eta} \cot{\paren{\frac{s-\eta}2}}=0$, we have
    \begin{align*}
    \begin{split}
        \paren{\p_s+\p_\eta} K(s,\eta)
        =&-\p_s \frac{\Delta\bm{X}\cdot \Delta\partial_s \bm{X}}{|\Delta\bm{X}|^2}I-2\partial_s \frac{\paren{\Delta\bm{X}\cdot \Delta\partial_s \bm{X}}\Delta\bm{X}\otimes\Delta\bm{X}}{|\Delta\bm{X}|^4}\\
        &+\partial_s \frac{\Delta\p_s\bm{X}\otimes\Delta\bm{X}}{|\Delta\bm{X}|^2}+\partial_s \frac{\Delta\bm{X}\otimes\p_s\Delta\bm{X}}{|\Delta\bm{X}|^2}.
    \end{split}
    \end{align*}
    Therefore, we obtain
    \begin{align*}
    \begin{split}
        \|I_2\|_{C^{0,\beta}}\leq C \frac{\|\partial_s\bm{X}\|_{C^1}\|\partial_s\bm{X}\|_{C^2}}{|\bm{X}|_*^2}\|\bm{u}\|_{C^{0}},
    \end{split}
    \end{align*}
    thus concluding the desired result. 
The estimate \eqref{Rest2} with $k=0$ is proved in \cite{MoriRodenbergSpirn19}, and combined with the estimates above show the result for $k=1$. 
    The cases $k\geq2$ follow by an induction procedure.

\end{proof}
\begin{lemma}[Semigroup Estimates]\label{SemigroupLem}
    Let $f\in C^{\alpha}(\mathbb{S}^1)$, $\alpha\in(0,1)$, and let $\beta\geq0$ be such that $0\leq\beta-\alpha\leq 3$. Then, there exists a constant $C=C(\alpha,\beta)$ such that
    \begin{equation*}
        \begin{aligned}
            \|e^{-t|\nabla|^3}f\|_{C^\beta}&\leq C t^{\frac{\alpha-\beta}{3}}\|f\|_{C^\alpha}.
        \end{aligned}
    \end{equation*}
    Moreover, if $f\in h^{1,\gamma}(\mathbb{S}^1)$, then $e^{-t|\nabla|^3}f \in X\paren{T}$, where $X(t)$ is defined in \eqref{space_fixed}, and
    \begin{equation*}
        \lim_{t\rightarrow 0^+} \norm{e^{-t|\nabla|^3}f-f}_{C^{1,\gamma}}=0.
    \end{equation*}
\end{lemma}

\begin{remark}
    Given $f\in h^{1,\gamma}(\mathbb{S}^1)$, then for all $\varepsilon>0$, there exists a function $f_\varepsilon\in C^\infty(\mathbb{S}^1)$ s.t. $\norm{f-f_\varepsilon}_{C^{1,\gamma}}\leq \varepsilon$.
    We can claim that 
    \begin{equation*}
        \lim_{t\rightarrow 0^+} \norm{e^{-t|\nabla|^3}f_\varepsilon-f_\varepsilon}_{C^{2}}=0.
    \end{equation*}
    Since $e^{-t|\nabla|^3}$ is bounded from $C^{1,\gamma}(\mathbb{S}^1)$ to itself, 
    we obtain the right continuity of  $e^{-t|\nabla|^3}$ at $t=0$ in $h^{1,\gamma}(\mathbb{S}^1)$.
    Next,  for all $\tilde{t}>0$,  $e^{-\tilde{t}|\nabla|^3}f\in C^\infty(\mathbb{S}^1)$, so $e^{-t|\nabla|^3}f \in C\paren{(\tilde{t},T); C^{4,\gamma}(\mathbb{S}^1) }$.
Also,  for all $t>0$, 
\begin{equation*}
    \begin{aligned}
            t\|e^{-t|\nabla|^3}f\|_{C^{4,\gamma}}&\leq C \|f\|_{C^{1,\gamma}}.
        \end{aligned}
\end{equation*}
  Assembling these results, we obtain  $e^{-t|\nabla|^3}f \in X\paren{T}$
\end{remark}

\section{Tension Estimates}\label{sec3}

Consider the equation \eqref{tension_eqn} written as follows:
\begin{equation}\label{ten_eq}
    \begin{aligned}
    \frac14|\nabla|\sigma+\mathcal{M}(\varphi)\sigma=F,
    \end{aligned}
\end{equation}
where
\begin{equation}\label{Meq}
    \begin{aligned}
        \mathcal{M}(\varphi)\sigma&=-\frac14\bm{\tau}(s)\cdot[\mathcal{H},\bm{n}](\sigma(1+\partial_s\varphi))+\frac14\bm{\tau}(s)\cdot[\mathcal{H},\bm{\tau}]\partial_s\sigma\\
        &\quad-\bm{\tau(s)}\cdot\mathcal{R}\big(-\sigma(1+\partial_s\varphi)\bm{n}(s)+(\partial_s\sigma)\bm{\tau}(s)\big).
    \end{aligned}
\end{equation}
In our problem, $F$ will be given by
\begin{equation}\label{Fdef}
    \begin{aligned}
        F(\varphi)=-\frac{1}{4}\bm{\tau}(s)\cdot[\mc{H},\bm{n}]\Big(\p_s^3\varphi +\frac{1}{2}(1+\p_s\varphi)^3\Big)\!+\!\bm{\tau}(s)\!\cdot\!\mc{R}\Big(\big(\p_s^3\varphi +\frac{1}{2}(1+\p_s\varphi)^3\big) \bm{n}(s)\Big).
    \end{aligned}
\end{equation}

\begin{prop}\label{Mprop}
    Let $\norm{\varphi}_{ C^{1}(\mathbb{S}^1)}\leq M_1$, $|\varphi|_*>m_1$. Then, $\mathcal{M}(\varphi)$ maps $C^1(\mathbb{S}^1)\rightarrow C^{0,\alpha}(\mathbb{S}^1)$ and
    \begin{equation*}
        \|\mathcal{M}(\varphi)\sigma\|_{C^{0,\alpha}}\leq C(M_1, m_1) \|\sigma\|_{C^1}. 
    \end{equation*}
    Moreover, for $0\leq k\in\mathbb{Z}$, if $\varphi\in C^{1+k}(\mathbb{S}^1)$, then
     \begin{equation*}
        \|(\mathcal{M}(\varphi)\sigma)\|_{C^{k,\alpha}}\leq  C(M_1, m_1)  \big((1+\|\partial_s\varphi\|_{C^{k}})\|\sigma\|_{C^1}+\|\sigma\|_{C^{1+k}}\big).
    \end{equation*}
    
\end{prop}

\begin{proof}
The estimates \eqref{comm1}, \eqref{Rest1} in Lemma \ref{CommutatorLem} and Proposition \ref{RestimateProp} give that
    \begin{equation*}
    \begin{aligned}
        \|\mathcal{M}(\varphi)\sigma\|_{C^{0,\alpha}}
        &\leq C (1+\|\p_s \varphi\|_{C^{0}})\\
        &\qquad\times\Big(1+\|\p_s \varphi\|_{C^{0}}+\frac{(1+\|\p_s \varphi\|_{C^{0}})^2}{|\varphi|_*^2}\Big)\big((1+\|\p_s \varphi\|_{C^{0}})\|\sigma\|_{C^0}+\|\p_s \sigma\|_{C^{0}}\big).
    \end{aligned}
    \end{equation*}
    Next, taking a derivative in \eqref{Meq}, we similarly obtain that
    \begin{align*}
        \begin{split}
                \|\p_s(\mathcal{M}(\varphi)\sigma)\|_{C^{0,\alpha}}
            &\leq C \norm{\p_s\bm{\tau}}_{C^{0,\alpha}}\paren{I_{01}+I_{02}+I_{03}}\\
            &\quad+C \norm{\bm{\tau}}_{C^{0,\alpha}}\paren{I_{11}+I_{12}+I_{13}},
        \end{split}
    \end{align*}
    where
    \begin{align*}
    \begin{split}
        I_{01}:=&\norm{\p_s\bm{n}}_{C^{0}}\norm{\sigma(1+\partial_s\varphi)}_{C^{0}}\\
        I_{02}:=&\norm{\p_s\bm{\tau}}_{C^{0}}\norm{\p_s\sigma}_{C^{0}}\\
        I_{03}:=&\frac{\|\partial_s\bm{X}\|_{C^1}^2}{|\varphi|_*^2}\norm{-\sigma(1+\partial_s\varphi)\bm{n}+(\partial_s\sigma)\bm{\tau}}_{C^{0}}\\
        I_{11}:=&\norm{\p_s^2\bm{n}}_{C^{0}}\norm{\sigma(1+\partial_s\varphi)}_{C^{0}}+\norm{\p_s\bm{n}}_{C^{0}}\norm{\p_s\paren{\sigma(1+\partial_s\varphi)}}_{C^{0}}\\
        I_{12}:=&\norm{\p_s^2\bm{\tau}}_{C^{0}}\norm{\p_s\sigma}_{C^{0}}+\norm{\p_s\bm{\tau}}_{C^{0}}\norm{\p_s^2\sigma}_{C^{0}}\\
        I_{13}:=&\frac{\|\partial_s\bm{X}\|_{C^1}}{|\varphi|_*^2}\left(\|\partial_s\bm{X}\|_{C^2}\|-\sigma(1+\partial_s\varphi)\bm{n}+(\partial_s\sigma)\bm{\tau}\|_{C^{0}}\right.\\
        &\qquad\qquad\left.+\|\partial_s\bm{X}\|_{C^1}\|\p_s\paren{-\sigma(1+\partial_s\varphi)\bm{n}+(\partial_s\sigma)\bm{\tau}}\|_{C^{0}}\right).
    \end{split}
    \end{align*}
The case $k\geq 2$ follows analogously by repeated differentiation and application of Lemma \ref{InterpolationLem} and  Proposition \ref{RestimateProp}.

\end{proof}

Let us write equation \eqref{ten_eq} as a perturbation of the identity:
\begin{equation}\label{eq_pert}
    \begin{aligned}
    \big(I+K(\varphi)\big)\sigma=\big(I+\frac14|\nabla|\big)^{-1}F,
    \end{aligned}
\end{equation}
where
\begin{equation}\label{K_def}
    \begin{aligned}
    K(\varphi)\sigma=\big(I+\frac14|\nabla|\big)^{-1}\big(\mathcal{M}(\varphi)-I\big)\sigma.
    \end{aligned}
\end{equation}
Then, Proposition \ref{Mprop} and the fact that $\big(I+\frac14|\nabla|\big)^{-1}$ is a bounded operator from $C^{0,\alpha}(\mathbb{S}^1)$ to $C^{1,\alpha}(\mathbb{S}^1)$ imply that $K(\varphi)$ is in fact a compact perturbation.
\begin{prop}[Tension Estimate]\label{TensionProp}
Let $\alpha\in(0,1)$, $\norm{\varphi}_{ C^{1}(\mathbb{S}^1)}\leq M_1$ not describing a circle, $|\varphi|_*>m_1$, and $F\in C^\alpha(\mathbb{S}^1)$. Then, there exists a unique solution to \eqref{ten_eq} and a constant $C_\varphi<\infty$ such that
\begin{equation*}
\begin{aligned}
    \|\sigma\|_{C^{1,\alpha}}&\leq C_\varphi\|F\|_{C^{0,\alpha}}. 
\end{aligned}
\end{equation*}
    Moreover, for $0\leq k\in\mathbb{Z}$,  if $\varphi\in C^{1+k}(\mathbb{S}^1)$ and $F\in C^{k,\alpha}(\mathbb{S}^1)$, then
    \begin{equation*}
        \begin{aligned}
              \|\sigma\|_{C^{1+k,\alpha}}&\leq C_\varphi C\big(\|F\|_{C^{k,\alpha}}+(1+\|\varphi\|_{C^{k+1}})\|\sigma\|_{C^1}\big),
        \end{aligned}
    \end{equation*}
    where $C=C(M_1, m_1)$.
\end{prop}
\begin{proof}
The first part is proved in \cite[Theorem 1.2]{KuoLaiMoriRodenberg2023}, where it is shown that  $I+K(\varphi)$ has trivial kernel when $\bm{X}$ is not a circle.
The second estimate is obtained as a consequence of the first one. In fact, taking a derivative in \eqref{ten_eq} we obtain
    \begin{equation*}
        \begin{aligned}
    \frac14|\nabla|(\partial_s\sigma)+\mathcal{M}(\varphi)(\partial_s\sigma)=\partial_s F-(\partial_s\mathcal{M}(\varphi))\sigma,
        \end{aligned}
    \end{equation*}
    where we are denoting
    \begin{equation*}
    \begin{aligned}
        (\partial_s\mathcal{M}(\varphi))\sigma&\!=\!-\frac14\partial_s\bm{\tau}(s)\!\cdot\![\mathcal{H},\bm{n}](\sigma(1+\partial_s\varphi))\!-\!\frac14\bm{\tau}(s)\!\cdot\![\mathcal{H},\partial_s\bm{n}](\sigma(1+\partial_s\varphi))\\
        &\quad-\frac14\bm{\tau}(s)\!\cdot\![\mathcal{H},\bm{n}](\sigma\partial_s^2\varphi)+\frac14\partial_s\bm{\tau}(s)\cdot[\mathcal{H},\bm{\tau}]\partial_s\sigma+\frac14\bm{\tau}(s)\cdot[\mathcal{H},\partial_s\bm{\tau}]\partial_s\sigma
        \\
        &\quad-\partial_s\bm{\tau(s)}\cdot\mathcal{R}\big(-\sigma(1+\partial_s\varphi)\bm{n}(s)+(\partial_s\sigma)\bm{\tau}(s)\big)\\
        &\quad-\bm{\tau(s)}\cdot(\partial_s\mathcal{R})\big(-\sigma(1+\partial_s\varphi)\bm{n}(s)+(\partial_s\sigma)\bm{\tau}(s)\big)\\
        &\quad-\bm{\tau(s)}\cdot\mathcal{R}\big(-\sigma\partial_s^2\varphi\bm{n}(s)-\sigma(1+\partial_s\varphi)\partial_s\bm{n}(s)+(\partial_s\sigma)\partial_s\bm{\tau}(s)\big).
    \end{aligned}
\end{equation*}
Therefore, 
\begin{equation*}
    \begin{aligned}
        \|\partial_s\sigma\|_{C^{1,\alpha}}&\leq C_\varphi \big(\|\partial_s F\|_{C^{0,\alpha}}+\|(\partial_s \mathcal{M}(\varphi))\sigma\|_{C^{0,\alpha}}\big),
    \end{aligned}
\end{equation*}
and proceeding as in Proposition \ref{Mprop},
\begin{equation*}
    \begin{aligned}
              \|\partial_s\sigma\|_{C^{1,\alpha}}&\leq C_\varphi \|\partial_s F\|_{C^{0,\alpha}}+C_{\varphi}C(\|\varphi\|_{C^1},|\varphi|_*)(1+\|\partial_s\varphi\|_{C^1})\|\sigma\|_{C^1}.
    \end{aligned}
\end{equation*}    
The case $k\geq2$ follows similarly by noticing that
\begin{equation*}
    \begin{aligned}
                \|\big(\partial_s^k \mathcal{M}(\varphi)\big)\sigma\|_{C^{0,\alpha}}&=\norm{\partial_s^k \big(\mathcal{M}(\varphi)\sigma\big)-\sum_{i=0}^{k-1}\begin{pmatrix}
                    k\\
                    i
                \end{pmatrix} \PPDD{i}{s}{}\mathcal{M}(\varphi)\partial_s^{k-i}\sigma}_{C^{0,\alpha}}\\
                &\leq C(\|\varphi\|_{C^1},|\varphi|_*)\big((1+\|\partial_s\varphi\|_{C^k})\|\sigma\|_{C^1}+(1+\|\partial_s\varphi\|_{C^1})\|\sigma\|_{C^k}\big),
    \end{aligned}
\end{equation*}
and performing interpolation on the last term
\begin{equation*}
    \begin{aligned}
        (1+\|\varphi\|_{C^2})\|\sigma\|_{C^k}&\leq (1+\|\varphi\|_{C^1}^{1-\frac{1}{k}}\|\varphi\|_{C^{1+k}}^{\frac{1}{k}})\|\sigma\|_{C^1}^{\frac{1}{k}}\|\sigma\|_{C^{1+k}}^{1-\frac{1}{k}}\\
        &\leq C(\varepsilon)(1+\|\varphi\|_{C^{1+k}})\|\sigma\|_{C^1}+\varepsilon (1+\|\varphi\|_{C^1})\|\sigma\|_{C^{1+k}},
    \end{aligned}
\end{equation*}
so that we can choose $\varepsilon=\varepsilon(C_\varphi, M_1,m_1)$ small enough to absorb the last term with the left-hand side.

\end{proof}
We will also need the continuity of $\sigma$ with respect to $\varphi$. Proceeding as in Propositions \ref{Mprop} and \ref{TensionProp}, we obtain the following result.
\begin{prop}\label{M_cont_prop}
    Let $\norm{\varphi_1}_{ C^{1}(\mathbb{S}^1)},\norm{\varphi_2}_{ C^{1}(\mathbb{S}^1)}\leq M_1$, $|\varphi_1|_*, |\varphi_2|_*\geq m_1>0$. Then,
    \begin{equation*}
        \|\big(\mathcal{M}(\varphi_2)-\mathcal{M}(\varphi_1)\big)\sigma\|_{C^{0,\alpha}}\leq C(M_1,m_1)\|\varphi_2-\varphi_1\|_{C^{1}}\|\sigma\|_{C^1}.
    \end{equation*}
Additionally, if $\varphi_1, \varphi_2\in C^1(\mathbb{S}^2)$, then
   \begin{equation*}
   \begin{aligned}
        \|\big(\p_s\mathcal{M}(\varphi_1)-\p_s\mathcal{M}&(\varphi_2)\big)\sigma\|_{C^{0,\alpha}}\\
        &\leq C\Big(\|\varphi_2-\varphi_1\|_{C^{2}}+\|\varphi_2-\varphi_1\|_{C^{1}}\big(1+\|\p_s\varphi_1\|_{C^{1}}+\|\p_s\varphi_2\|_{C^{1}}\big)\Big)\|\sigma\|_{C^1},
   \end{aligned}
    \end{equation*}
    for a constant $C=C(M_1,m_1)$.
\end{prop}
\begin{prop}\label{ten_cont_prop}
Let $\sigma_1$, $\sigma_2$ be the corresponding solutions to \eqref{ten_eq} with $|\varphi_1|_*>0$, $|\varphi_2|_*\geq m_1 >0$,  $\norm{\varphi_1}_{ C^{1}(\mathbb{S}^1)},\norm{\varphi_2}_{ C^{1}(\mathbb{S}^1)}\leq M_1$, $F_1, F_2\in C^\alpha(\mathbb{S}^1))$. Then, 
\begin{equation}\label{cont_ten}
\begin{aligned}
    \|\sigma_1-\sigma_2\|_{C^{1,\alpha}}&\leq C_{\varphi_1}C_{\varphi_2}C \|\varphi_2-\varphi_1\|_{C^{1}}\|F_1\|_{C^{0,\alpha}}+C_{\varphi_2}\|F_1-F_2\|_{C^{0,\alpha}},
\end{aligned}
\end{equation}
with $C=C(M_1, m_1)$.
Moreover, if $\varphi_1,\varphi_2\in C^2(\mathbb{S}^1)$,  $F_1, F_2\in C^{1,\alpha}(\mathbb{S}^1))$, then
\begin{equation*}
\begin{aligned}
    \|\sigma_1-\sigma_2\|_{C^{2,\alpha}}
    &\leq C_{\varphi_1}C_{\varphi_2}C \|\varphi_2-\varphi_1\|_{C^{1}}\big(\|F_1\|_{C^{1,\alpha}}+(1+\|\p_s\varphi_1\|_{C^{1}})\|\sigma_1\|_{C^{1}}\big)\\
    &\quad+ C_{\varphi_2}\|F_1-F_2\|_{C^{1,\alpha}}\!+\!C_{\varphi_2}C\|\varphi_1-\varphi_2\|_{C^2}\|\sigma_2\|_{C^1}\!\\
    &\quad+\!C_{\varphi_2}C(1+\|\p_s\varphi\|_{C^1})\|\sigma_1-\sigma_2\|_{C^1}\\
    &\quad+C_{\varphi}C\|\varphi_1-\varphi_2\|_{C^1}\Big(1+\|\p_s\varphi_1\|_{C^{1}}+\|\p_s\varphi_2\|_{C^{1}}\Big)\|\sigma_2\|_{C^{1}},
\end{aligned}
\end{equation*}
where $C=C(M_1,m_1)$.
\end{prop}
\begin{proof}
    Let us write  $\tilde{F}_j=\big(I+\frac14|\nabla|\big)^{-1}F_j$  with $j=1,2$, and
    \begin{equation*}
        \sigma_j=\big(I+K(\varphi_j)\big)^{-1}\tilde{F}_j.
    \end{equation*}
Then,
\begin{equation*}
\begin{aligned}
  \sigma_1-\sigma_2&= \big(I+K(\varphi_1)\big)^{-1}\big(K(\varphi_2)-K(\varphi_1)\big)\big(I+K(\varphi_2)\big)^{-1}\tilde{F}_1\\
  &\quad+\big(I+K(\varphi_2)\big)^{-1}(\tilde{F}_1-\tilde{F}_2),
\end{aligned}
\end{equation*}
and substituting the expression of $K(\varphi_j)$ \eqref{K_def},
\begin{equation*}
\begin{aligned}
  \sigma_1-\sigma_2&= \big(I+K(\varphi_1)\big)^{-1}\big(I+\frac14|\nabla|\big)^{-1}\big(\mathcal{M}(\varphi_2)-\mathcal{M}(\varphi_1)\big)\big(I+K(\varphi_2)\big)^{-1}\tilde{F}_1\\
  &\quad+\big(I+K(\varphi_2)\big)^{-1}(\tilde{F}_1-\tilde{F}_2).
\end{aligned}
\end{equation*}
Thus, Propositions \ref{TensionProp} and \ref{M_cont_prop} give the first inequality.

The second one is then a consequence of the previous one. In fact, we can write
    \begin{equation*}
        \partial_s\sigma_j=\big(I+K(\varphi_j)\big)^{-1}\big(I+\frac14|\nabla|\big)^{-1}Z_j, \qquad Z_j=\p_s F_j-(\p_s\mc{M}(\varphi_j))\sigma_j,
    \end{equation*}
thus we have that
\begin{equation*}
\begin{aligned}
    \|\p_s\sigma_1-\p_s\sigma_2\|_{C^{1,\alpha}}&\leq C_{\varphi_1}C_{\varphi_2}C \|\varphi_2-\varphi_1\|_{C^{1}}\|Z_1\|_{C^{0,\alpha}}+C_{\varphi_2}\|Z_1-Z_2\|_{C^{0,\alpha}}.
\end{aligned}
\end{equation*}
Therefore, 
\begin{equation*}
\begin{aligned}
    \|\p_s\sigma_1-\p_s\sigma_2&\|_{C^{1,\alpha}}\leq C_{\varphi_1}C_{\varphi_2}C \|\varphi_2-\varphi_1\|_{C^{1}}\big(\|\p_s F_1\|_{C^{0,\alpha}}+(1+\|\p_s\varphi_1\|_{C^{1}})\|\sigma_1\|_{C^{1}}\big)\\
    &\quad+C_{\varphi_2}\|\p_s F_1-\p_s F_2\|_{C^{0,\alpha}}+C_{\varphi_2}\|\p_s \mc{M}(\varphi_1)(\sigma_1)-\p_s \mc{M}(\varphi_2)(\sigma_2)\|_{C^{0,\alpha}}.
\end{aligned}
\end{equation*}
By writing
\begin{equation*}
\begin{aligned}
    \p_s \mc{M}(\varphi_1)(\sigma_1)-\p_s \mc{M}(\varphi_2)(\sigma_2)=    (\p_s \mc{M}(\varphi_1))(\sigma_1-\sigma_2)+\big(\p_s \mc{M}(\varphi_1)-\p_s \mc{M}(\varphi_2)\big)(\sigma_2),
\end{aligned}
\end{equation*}
Propositions \ref{Mprop} and \ref{M_cont_prop} provide that
\begin{equation*}
\begin{aligned}
    \|\p_s \mc{M}(\varphi_1)(\sigma_1)-&\p_s \mc{M}(\varphi_2)(\sigma_2)\|_{C^{0,\alpha}}\\
    &\leq C(1+\|\p_s\varphi_1\|_{C^1})\|\sigma_1-\sigma_2\|_{C^1}\\
    &\quad+C\Big(\|\varphi_2-\varphi_1\|_{C^{2}}+\|\varphi_2-\varphi_1\|_{C^{1}}\big(1+\|\p_s\varphi_1\|_{C^{1}}+\|\p_s\varphi_2\|_{C^{1}}\big)\Big)\|\sigma_2\|_{C^1},
\end{aligned}
\end{equation*}
hence we conclude the proof.

\end{proof}

When $F=F(\varphi)$ given by \eqref{Fdef}, the previous propositions provide the estimates of the tension $\sigma$ in terms of the interface $\varphi$. As a straightforward application of estimates \eqref{comm1}, \eqref{Rest1} from Lemma \ref{CommutatorLem} and Proposition \ref{RestimateProp} we obtain the the following estimates for $F(\varphi)$.
\begin{prop}\label{ForceEst}
Let $0\leq k\in\mathbb{Z}$, $\alpha\in(0,1)$, $|\varphi|_*>0$, and $F(\varphi)$ given by \eqref{Fdef}. Then, if $\varphi\in C^{3}(\mathbb{S}^1)$,
\begin{equation*}
\begin{aligned}
   \|F(\varphi)\|_{C^{0,\alpha}}&\leq C(\|\partial_s\varphi\|_{C^0},|\varphi|_*) \big(1+\|\partial_s^3\varphi\|_{C^0}\big).
\end{aligned}
\end{equation*}
Moreover, for $\varphi\in C^{3+k}(\mathbb{S}^1)$
   \begin{equation*}
        \begin{aligned}
           \|F(\varphi)\|_{C^{k,\alpha}}&\leq C(\|\partial_s\varphi\|_{C^0},|\varphi|_*) \Big((1+\|\varphi\|_{C^{1+k}})\big(1+\|\varphi\|_{C^3}\big)+\|\varphi\|_{C^{3+k}}\Big).
        \end{aligned}
    \end{equation*}
For $\varphi_1,\varphi_2$ satisfying the same conditions as $\varphi$,
\begin{equation*}
\begin{aligned}
   \|F(\varphi_1)-F(\varphi_2)\|_{C^{0,\alpha}}&\leq C(\|\partial_s\varphi_1\|_{C^0},\|\partial_s\varphi_2\|_{C^0},|\varphi_1|_*,|\varphi_2|_*)\\
   &\quad\times \Big(\|\partial_s\varphi_1-\partial_s\varphi_2\|_{C^0}\big(1+\|\partial_s^3\varphi_1\|_{C^0}\big)+\|\partial_s^3\varphi_1-\partial_s^3\varphi_2\|_{C^0}\Big),
\end{aligned}
\end{equation*}
\begin{equation*}
    \begin{aligned}
        \|\p_s F(\varphi_1)-&\p_s F(\varphi_2)\|_{C^{0,\alpha}}\leq C\|\varphi_1-\varphi_2\|_{C^4}+C(1+\|\varphi_2\|_{C^2})\|\varphi_1-\varphi_2\|_{C^3}\\
        &\quad+C\|\varphi_1-\varphi_2\|_{C^2}\big(1+\|\p_s^3\varphi_1\|_{C_0}\big)\\
        &\quad+C\|\varphi_1-\varphi_2\|_{C^1}\Big(\big(1+\|\varphi_1\|_{C^2}+\|\varphi_2\|_{C^2}\big)\big(1+\|\p_s^3\varphi_1\|_{C_0}\big)+1+\|\p_s^4\varphi_1\|_{C_0}\Big),
    \end{aligned}
\end{equation*}
with $C=C(\|\partial_s\varphi_1\|_{C^0},\|\partial_s\varphi_2\|_{C^0},|\varphi_1|_*,|\varphi_2|_*)$.
\end{prop}
Plugging in the estimates for $F(\varphi)$ into Propositions \ref{TensionProp} and \ref{ten_cont_prop}, we conclude the estimates for the tension $\sigma$.
\begin{prop}\label{TensionProp2}
Let $0\leq k\in\mathbb{Z}$, $\alpha\in(0,1)$, $\varphi\in C^{3}(\mathbb{S}^1)$ with $\norm{\varphi}_{ C^{1}(\mathbb{S}^1)}\leq M_1$ and not describing a circle, $|\varphi|_*\geq m_1>0$, and $F(\varphi)$ given by \eqref{Fdef}. Then, there exists a unique solution to \eqref{ten_eq} and
\begin{equation*}
\begin{aligned}
    \|\sigma\|_{C^{1,\alpha}}&\leq C_\varphi \,C(M_1,m_1) \big(1+\|\partial_s^3\varphi\|_{C^0}\big). 
\end{aligned}
\end{equation*}
Moreover, if $\varphi\in C^{3+k}(\mathbb{S}^1)$,
    \begin{equation}\label{sigma_kest}
        \begin{aligned}
              \|\sigma\|_{C^{1+k,\alpha}}&\leq C_\varphi \,C(M_1,m_1) \Big((1+\|\varphi\|_{C^{1+k}})\big(1+\|\varphi\|_{C^3}\big)+\|\varphi\|_{C^{3+k}}\Big).
        \end{aligned}
    \end{equation}
For $\varphi_1, \varphi_2$ satisfying the same conditions as $\varphi$, then
\begin{equation*}
\begin{aligned}
    \|\sigma_1(\varphi_1)-\sigma_2&(\varphi_2)\|_{C^{1,\alpha}}\leq  C(M_1,m_1)\\
    &\times\Big(C_{\varphi_2}(1+C_{\varphi_1}) \|\partial_s\varphi_1-\partial_s\varphi_2\|_{C^0}\big(1+\|\partial_s^3\varphi_1\|_{C^0}\big)+C_{\varphi_2}\|\partial_s^3\varphi_1-\partial_s^3\varphi_2\|_{C^0} \Big).
\end{aligned}
\end{equation*}
and
    \begin{equation*}
\begin{aligned}
    \|\sigma_1&(\varphi_1)-\sigma_2(\varphi_2)\|_{C^{2,\alpha}}\leq  C M_2^3\|\varphi_1\!-\!\varphi_2\|_{C^4}\!+\!C M_2^3\big(1+\|\varphi_1\|_{C^2}\!+\!\|\varphi_2\|_{C^2}\big)\|\varphi_1\!-\!\varphi_2\|_{C^3}\\
        &+C M_2^3\|\varphi_1\!-\!\varphi_2\|_{C^2}\big(1+\|\p_s^3\varphi_1\|_{C_0}\!+\!\|\p_s^3\varphi_2\|_{C_0}\big)\\
        &+C M_2^3\|\varphi_1\!-\!\varphi_2\|_{C^1}\Big(\big(1+\|\varphi_1\|_{C^2}\!+\!\|\varphi_2\|_{C^2}\big)\big(1\!+\!\|\p_s^3\varphi_1\|_{C_0}\!+\!\|\p_s^3\varphi_2\|_{C_0}\big)\!+\!1\!+\!\|\p_s^4\varphi_1\|_{C_0}\Big),
\end{aligned}
\end{equation*}
where $C=C(M_1,m_1)$ and $M_2=\max \{C_{\varphi_1},C_{\varphi_2},1\}$.
\end{prop}
Now, set 
\begin{equation*}
    \mc{O}:=\{\varphi\in  C^{1}(\mathbb{S}^1):  |\varphi|_* >0, d\paren{\varphi}>0\},
\end{equation*}
We have the following proposition about the continuity of $C_\varphi$ in $\mc{O}$.

\begin{prop}\label{Tensionbdn}
    $\paren{I+K(\varphi)}^{-1}$ is continuous in $\varphi$ in $\mc{O}$ with respect to the $C^{1}\paren{\mbs}$ norm. 
\end{prop}
\begin{proof}
    Given $\varphi_0\in \mc{O}$, $\tilde{\mc{U}}_{\varphi_0}:=\set{\varphi\in  C^{1}(\mathbb{S}^1)}{\norm{\varphi}_{C^{1}(\mathbb{S}^1)}<2\norm{\varphi_0}_{C^{1}(\mathbb{S}^1)},|\varphi|_* >\frac{1}{2}|\varphi_0|_*,d\paren{\varphi}>0}$ is a neighborhood of $\varphi_0$ in $\mc{O}$.
    By Proposition \ref{M_cont_prop} and the fact that $\big(I+\frac14|\nabla|\big)^{-1}$ is a bounded operator from $C^{0,\alpha}(\mathbb{S}^1)$ to $C^{1,\alpha}(\mathbb{S}^1)$, $K(\varphi)$ is continuous at $\varphi_0$ in $C^1\paren{\mbs}$ norm.
    Therefore, there exists a neighborhood $\mc{O}_{\varphi_0}$ of $\varphi_0$ s.t. for all $\varphi\in \mc{O}_{\varphi_0}$,
    \begin{equation*}
        \norm{K(\varphi_0)-K(\varphi)}_{C^{1,\alpha}(\mathbb{S}^1)}\leq \frac{1}{2}\norm{\paren{I+K(\varphi_0)}^{-1}}_{C^{1,\alpha}(\mathbb{S}^1)}^{-1}.
    \end{equation*}
    Set $A=I+K(\varphi_0)$ and $B=K(\varphi_0)-K(\varphi)$, we obtain $\norm{A^{-1}B}\leq \norm{A^{-1}}\norm{B}\leq \frac{1}{2}$
    and
    \begin{equation*}
    \begin{split}
        &\paren{I+K(\varphi)}^{-1}-\paren{I+K(\varphi_0)}^{-1}
        =\paren{A-B}^{-1}-A^{-1}\\
        =&\paren{\paren{I-A^{-1}B}^{-1}-I}A^{-1}
        =\paren{\sum_{k=1}^\infty \paren{A^{-1}B}^k}A^{-1},
    \end{split}
    \end{equation*}
    where the last equality follows by expanding as a Neumann series.
    Therefore, as $\norm{\varphi-\varphi_0}_{C^{1}(\mathbb{S}^1)}\rightarrow 0$,  $\norm{B}\rightarrow 0$, and
    \begin{equation*}
        \norm{\paren{A-B}^{-1}-A^{-1}}\leq \frac{\norm{A^{-1}}^2\norm{B}}{1-\norm{A^{-1}}\norm{B}} \rightarrow 0.
    \end{equation*}
\end{proof}

\section{Local Well-Posedness}\label{sec4}

We begin by denoting the nonlinearity in \eqref{theta_eqn} as follows
\begin{equation}\label{Nonlinear_split}
\begin{aligned}
\mathcal{N}(\varphi,\sigma)=\mathcal{N}_1(\varphi,\sigma)+\mathcal{N}_2(\varphi,\sigma)+\mathcal{N}_2(\varphi,\sigma),
\end{aligned}
\end{equation}
\begin{equation*}
\begin{split}
\mathcal{N}_1(\varphi,\sigma)&=\frac{1}{4}\mc{H}\paren{\frac{1}{2}(\p_s\varphi)^3-\sigma(1+\p_s\varphi)},\\
\mathcal{N}_2(\varphi,\sigma)&=-\frac{1}{4}\bm{n}(s)\cdot[\mc{H},\bm{n}]\Big(\p_s^3\varphi +\frac{1}{2}(1+\p_s\varphi)^3-\sigma(1+\p_s\varphi)\Big)-\frac{1}{4}\bm{n}(s)\cdot[\mc{H},\bm{\tau}](\partial_s\sigma),\\
\mathcal{N}_3(\varphi,\sigma)&=\bm{n}(s)\cdot\mc{R}\paren{\paren{\p_s^3\varphi +\frac{1}{2}(1+\p_s\varphi)^3-\sigma(1+\p_s\varphi)} \bm{n}(s)+(\p_s\sigma)\bm{\tau}(s)},
\end{split}
\end{equation*}
so that \eqref{theta_eqn} rewrites as
\begin{equation}\label{theta_eq_sum}
    \begin{aligned}
        \partial_t\varphi(t)+\frac14|\nabla|^3\varphi(t)&=\mathcal{N}(\varphi,\sigma)(t).
    \end{aligned}
\end{equation}
Let us write the equation in Duhamel's form:
\begin{equation}\label{duhamel_eq}
    \begin{aligned}
        \varphi(t)&=e^{-\frac14|\nabla|^3t}\varphi_0+\int_0^t e^{-\frac14|\nabla|^3(t-\tau)}\mathcal{N}(\varphi,\sigma)(\tau)d\tau.
    \end{aligned}
\end{equation}
Proposition \ref{TensionProp2} shows that for fixed $t$ we can solve the tension equation \eqref{tension_eqn}, $\sigma=\sigma(\varphi)$, hence we will write
\begin{equation*}
    \begin{aligned}
        \varphi(t)&=e^{-\frac14|\nabla|^3t}\varphi_0+\int_0^t e^{-\frac14|\nabla|^3(t-\tau)}\tilde{\mathcal{N}}(\varphi)(\tau)d\tau,
    \end{aligned}
\end{equation*}
with $\tilde{\mathcal{N}}(\varphi):=\mathcal{N}(\varphi,\sigma(\varphi))$. Then, we will perform a fixed point argument in the space $X(T)$ defined in \eqref{space_fixed}
and initial datum $\varphi\in h^{1,\gamma}(\mathbb{S}^1)$, $\gamma\in(0,1)$.

\subsection{Nonlinear Estimates}\label{nonlin}

In the following, whenever $t>0$ is fixed, we suppress the dependence on $t$ for clarity in notation.
\begin{prop}\label{force_est}
    Let $\gamma\in(0,1)$, $\varphi\in C^{4,\gamma}(\mathbb{S}^1)$, $\sigma\in C^{2}(\mathbb{S}^1)$, $0\leq k\in\mathbb{Z}$, and $F_n(\varphi,\sigma)$, $F_{\tau}(\varphi,\sigma)$ be given by \eqref{NormalTanForce}. Assume that
    \begin{equation}\label{bound_cond}
        \|\varphi\|_{C^{1,\gamma}}\leq M_1.
    \end{equation}
    Then,
    \begin{equation*}
        \begin{aligned}
            \|F_{n}\|_{C^0}&\leq C\Big(1+\|\varphi\|_{C^{3}}+\|\sigma\|_{C^0}\Big),\\
            \|F_{n}\|_{C^1}&\leq C\Big(1+\|\varphi\|_{C^{4}}+\paren{1+\|\sigma\|_{C^0}}\|\varphi\|_{C^{2}}+\|\sigma\|_{C^1}\Big),\\
            \|F_{\tau}\|_{C^k}&\leq \|\sigma\|_{C^{1+k}},
        \end{aligned}
    \end{equation*}
    and
    \begin{equation*}
        \begin{aligned}
            \|F_{1,n}-F_{2,n}\|_{C^0}&\leq C\Big(\|\varphi_1-\varphi_2\|_{C^{3}}+\|\varphi_1-\varphi_2\|_{C^1}+\|\sigma_1-\sigma_2\|_{C^0}\\
            &\qquad\quad+\|\sigma_2\|_{C^0}\|\varphi_1-\varphi_2\|_{C^1}\Big),\\
            \|F_{1,n}-F_{2,n}\|_{C^1}&\!\leq\! C\Big(\|\varphi_1\!-\!\varphi_2\|_{C^{4}}\!+\!\big(1\!+\!\|\sigma\|_{C^0}\big)\|\varphi_1\!-\!\varphi_2\|_{C^{2}}\\
            &\qquad+\big(1+\|\varphi_1\|_{C^2}+\|\varphi_2\|_{C^2}\big)\|\varphi_1-\varphi_2\|_{C^{1}}+\|\sigma_1-\sigma_2\|_{C^1}\\
            &\qquad+\|\sigma_1-\sigma_2\|_{C^0}(1+\|\varphi_1\|_{C^2})+\|\sigma_2\|_{C^1}\|\varphi_1-\varphi_2\|_{C^1}\Big),\\
            \|F_{1,\tau}-F_{2,\tau}\|_{C^k}&\leq \|\sigma_1-\sigma_2\|_{C^{1+k}}.
        \end{aligned}
    \end{equation*}
\end{prop}
\begin{proof}
A direct computations shows that
    \begin{equation*}
        \begin{aligned}
            \|F_{n}\|_{C^0}&\leq 1+\|\varphi\|_{C^{3}}+\frac12\|\varphi\|_{C^1}^3+\|\sigma\|_{C^0}(1+\|\varphi\|_{C^1}),\\
            \|F_{n}\|_{C^1}&\leq 1+\|\varphi\|_{C^{4}}+\big(\frac32(1+\|\varphi\|_{C^1})^2+\|\sigma\|_{C^0}\big)\|\varphi\|_{C^{2}}+\|\sigma\|_{C^1}(1+\|\varphi\|_{C^1}),\\
            \|F_{\tau}\|_{C^k}&\leq \|\sigma\|_{C^{1+k}},
        \end{aligned}
    \end{equation*}
    and
    \begin{equation*}
        \begin{aligned}
            \|F_{1,n}-F_{2,n}\|_{C^0}&\leq \|\varphi_1-\varphi_2\|_{C^{3}}+\|\varphi_1-\varphi_2\|_{C^1}\big(1+\|\varphi_1\|_{C^1}^2+\|\varphi_2\|_{C^1}^2\big)\\
            &\quad+\|\sigma_1-\sigma_2\|_{C^0}(1+\|\varphi_1\|_{C^1})+\|\sigma_2\|_{C^0}\|\varphi_1-\varphi_2\|_{C^1},\\
            \|F_{1,n}-F_{2,n}\|_{C^1}&\leq \|\varphi_1-\varphi_2\|_{C^{4}}+\big(1+\|\varphi_1\|_{C^1}^2+\|\varphi_2\|_{C^1}^2+\|\sigma\|_{C^0}\big)\|\varphi_1-\varphi_2\|_{C^{2}}\\
            &\quad+\big((1+\|\varphi_1\|_{C^1})\|\varphi_1\|_{C^2}+(1+\|\varphi_2\|_{C^1})\|\varphi_2\|_{C^2}\big)\|\varphi_1-\varphi_2\|_{C^{1}}\\
            &\quad+\|\sigma_1-\sigma_2\|_{C^1}(1+\|\varphi_1\|_{C^1})+\|\sigma_1-\sigma_2\|_{C^0}(1+\|\varphi_1\|_{C^2})+\|\sigma_2\|_{C^1}\|\varphi_1-\varphi_2\|_{C^1}.
        \end{aligned}
    \end{equation*}
Then, the proof is concluded by condition \eqref{bound_cond}.

\end{proof}

\begin{prop}\label{nonlinear_est}
    Let $T\in[0,1]$, $t\in(0,T]$, $\gamma\in(0,1)$, $\varepsilon\in \lclose{ 0,\min\{1-\gamma,\gamma\}}$, and $\varphi\in X(T)$, $\sigma(t)\in C^{2}(\mathbb{S}^1)$ for $t\in(0,T]$. Assume that
    \begin{equation}\label{AssumpBound}
    \|\varphi\|_{X(T)}\leq M_1, \abs{\varphi}_*\geq m_1>0.
    \end{equation}
    Then, it holds that for any $\delta\in(0,\gamma)$
    \begin{equation*}
        \begin{aligned}
            t^\frac{\varepsilon}{3}\|\mathcal{N}(\varphi,\sigma)(t)\|_{C^{1,\gamma+\varepsilon}}&\leq C\Big(t^{-\frac{3-\paren{\gamma-\delta}}{3}}+t^{-\frac{1}{3}}\|\sigma\|_{C^{0}}+t^{-\frac{1-\paren{\gamma-\delta}}{3}}\|\sigma\|_{C^{1}}+t^\frac{\varepsilon}{3}\|\sigma\|_{C^{1,\gamma+\varepsilon}}+\|\sigma\|_{C^{2}}\Big)
        \end{aligned}
    \end{equation*}
    where $C$ depends on $M_1, m_1$ and $\delta$.
\end{prop}
\begin{proof}
 We consider $\mathcal{N}_1$ in \eqref{Nonlinear_split} first:
\begin{equation*}
    \begin{aligned}
        \|\mathcal{N}_1(\varphi,\sigma)\|_{C^{1,\gamma+\varepsilon}}&\leq C\Big(\|(1+\partial_s\varphi)^3\|_{C^{1,\gamma+\varepsilon}}+\|\sigma(1+\partial_s\varphi)\|_{C^{1,\gamma+\varepsilon}}\Big)\\
        &\leq C\Big(1+\|\partial_s\varphi\|_{C^{1,\gamma+\varepsilon}}\|\partial_s\varphi\|_{C^{0}}^2+\|\sigma\|_{C^{1,\gamma+\varepsilon}}(1+\|\partial_s\varphi\|_{C^{0}})\\
        &\quad +\|\sigma\|_{C^{0}}(1+\|\partial_s\varphi\|_{C^{1,\gamma+\varepsilon}})\Big).
    \end{aligned}
\end{equation*}
and by the interpolation \eqref{interplationbddeq03} and the assumption \eqref{AssumpBound}
\begin{equation}\label{N1bound}
    \begin{aligned}
        t^\frac{\varepsilon}{3}\|\mathcal{N}_1(\varphi,\sigma)\|_{C^{1,\gamma+\varepsilon}}&\leq C\Big(t^{-\frac{1}{3}}\paren{1+\|\sigma\|_{C^{0}}}+t^\frac{\varepsilon}{3}\|\sigma\|_{C^{1,\gamma+\varepsilon}}\Big).
    \end{aligned}
\end{equation}
Next, we estimate $\|\mathcal{N}_2(\varphi,\sigma)\|_{C^{1,\gamma+\varepsilon}}$. Using the notation \eqref{NormalTanForce},
\begin{equation*}
    \begin{aligned}
\|\mathcal{N}_2(\varphi,\sigma)&\|_{C^{1,\gamma+\varepsilon}}
\leq C(1+\|\varphi\|_{C^{1,\gamma+\varepsilon}})\Big(\|[\mathcal{H},\bm{n}]F_n(\varphi,\sigma)\|_{C^{1,\gamma+\varepsilon}}\!+\!\|[\mathcal{H},\bm{\tau}]F_\tau(\varphi,\sigma)\|_{C^{1,\gamma+\varepsilon}}\Big),
    \end{aligned}
\end{equation*}
so Lemma \ref{CommutatorLem} gives that
\begin{equation*}
    \begin{aligned}
\|\mathcal{N}_2(\varphi,\sigma)\|_{C^{1,\gamma+\varepsilon}}
&\leq C(1+\|\varphi\|_{C^{1,\gamma+\varepsilon}})\Big((1+\|\varphi\|_{C^{2}})\|F_n\|_{C^{0}}+(1+\|\varphi\|_{C^{1}})\|F_n\|_{C^{1}}\\
&\quad\hspace{1.5cm}+(1+\|\varphi\|_{C^{2}})\|F_\tau\|_{C^{0}}+(1+\|\varphi\|_{C^{1}})\|F_\tau\|_{C^{1}}\Big).
    \end{aligned}
\end{equation*}
We substitute the estimates from Proposition \ref{force_est} and then use the inequality \eqref{interplationbddeq03} and the assumption \eqref{AssumpBound} to obtain for $\delta\in (0,\gamma)$
\begin{equation*}
    \begin{aligned}
t^\frac{\varepsilon}{3}\|\mathcal{N}_2(\varphi,\sigma)\|_{C^{1,\gamma+\varepsilon}}
&\leq C t^\frac{\varepsilon}{3}(1+\|\varphi\|_{C^{1,\gamma+\varepsilon}})\Big((1+\|\varphi\|_{C^{2}})\big(1+\|\varphi\|_{C^{3}}+\|\sigma\|_{C^0}\big)\\
&\qquad+(1+\|\varphi\|_{C^{1}})\big(1+\|\varphi\|_{C^{4}}+\paren{1+\|\sigma\|_{C^0}}\|\varphi\|_{C^{2}}+\|\sigma\|_{C^1}\big)\\
&\qquad+(1+\|\varphi\|_{C^{2}})\|\sigma\|_{C^{1}}+(1+\|\varphi\|_{C^{1}})\|\sigma\|_{C^{2}}\Big)\\
&\leq C \left( t^{-\frac{1-\paren{\gamma-\delta}}{3}}+t^{-\frac{2-2\paren{\gamma-\delta}}{3}}+t^{-\frac{3-\paren{\gamma-\delta}}{3}}+t^{-\frac{1-\paren{\gamma-\delta}}{3}}\norm{\sigma}_{C^0}\right.\\
&\qquad\left.+ \paren{1+t^{-\frac{1-\paren{\gamma-\delta}}{3}}}\norm{\sigma}_{C^1}+\norm{\sigma}_{C^2}\right) .
    \end{aligned}
\end{equation*}
hence, since $0<t\leq T\leq 1$, we conclude that for any $\delta\in(0,\gamma)$
\begin{equation*}
    \begin{aligned}
t^\frac{\varepsilon}{3}\|\mathcal{N}_2(\varphi,\sigma)\|_{C^{1,\gamma+\varepsilon}}&\leq  C\Big(t^{-\frac{3-\paren{\gamma-\delta}}{3}}+t^{-\frac{1-\paren{\gamma-\delta}}{3}}\|\sigma\|_{C^{1}}+\|\sigma\|_{C^{2}}\Big).
    \end{aligned}
\end{equation*}
Finally, the estimate for $\mathcal{N}_3(\varphi,\sigma)$ is obtained applying Proposition \ref{RestimateProp}, 
hence since $|\varphi|_*>0$, $\mathcal{N}_3(\varphi,\sigma)$ satisfies the same bound as $\mathcal{N}_2(\varphi,\sigma)$.

\end{proof}

Next, we solve for $\sigma=\sigma(\varphi)$ for fixed time and rewrite the nonlinear estimates solely in terms of $\varphi(t)$. 
\begin{prop}\label{nonli_est}
    Let $T\in[0,1]$, $t\in(0,T]$, $\gamma\in(0,1)$, $\varepsilon\in \lclose{ 0,\min\{1-\gamma,\gamma\}}$, $\delta\in(0,\gamma)$, and $|\bm{X}(t)|_*>0$, $\varphi\in X(T)$, and $\sigma(t)=\sigma(\varphi(t))$ the unique solution to \eqref{ten_eq} with $F=F(\varphi)$ given by \eqref{Fdef}. Assume that
    \begin{equation}\label{AssumpBound2}
    \|\varphi\|_{X(T)}\leq M_1, \abs{\varphi}_*\geq m_1>0, d\paren{\varphi}\geq m_2>0 .
    \end{equation}
    Then, it holds that for any $\delta\in(0,\gamma)$
    \begin{equation*}
        \begin{aligned}
            t^\frac{\varepsilon}{3}\|\mathcal{N}(\varphi,\sigma)(t)\|_{C^{1,\gamma+\varepsilon}}&\leq C M_2 t^{\frac{\gamma-\delta}{3}-1}.
        \end{aligned}
    \end{equation*}
    where $C$ depends on $\gamma,\varepsilon, \delta, M_1, m_1$ and $M_2=\max \{\sup_{t\in [0,T]} C_{\varphi\paren{t}},1\}$
\end{prop}
\begin{proof} We introduce the estimates in Proposition \ref{TensionProp2} into the one in Proposition \ref{nonlinear_est} and use that $0<t\leq 1$:
    \begin{equation*}
        \begin{aligned}
            t^\frac{\varepsilon}{3}\|\mathcal{N}(\varphi,\sigma)(t)\|_{C^{1,\gamma+\varepsilon}}&\leq C M_2\Big(t^{-\frac{3-\paren{\gamma-\delta}}{3}}+t^{-\frac{1}{3}}\|\partial_s^3\varphi\|_{C^{0}}+\|\varphi\|_{C^{4}}+\|\varphi\|_{C^{2}}\|\varphi\|_{C^{3}}\Big),
        \end{aligned}
    \end{equation*}
    Then, by the inequality \eqref{interplationbddeq04} and assumption \eqref{AssumpBound2},
    \begin{equation*}
        t^\frac{\varepsilon}{3}\|\mathcal{N}(\varphi,\sigma)(t)\|_{C^{1,\gamma+\varepsilon}}\leq C\paren{\gamma,\varepsilon, M_1, m_1} M_2 t^{\frac{\gamma-\delta}{3}-1}.
    \end{equation*}
\end{proof}

\subsection{Contraction Estimates}\label{contest}

We will also need contraction estimates on the nonlinear terms.

\begin{prop}\label{Contract_est}
       Let $T\in[0,1]$, $t\in(0,T]$, $\gamma\in(0,1)$, $\varepsilon\in \lclose{ 0,\min\{1-\gamma,\gamma\}}$, and $|\varphi_i(t)|_*>0$, $\varphi_i\in X(T)$, and for $i=1,2$, $\sigma_i(t)=\sigma(\varphi_i(t))$ the unique solution to \eqref{ten_eq} with $F=F(\varphi_i)$ given by \eqref{Fdef}. Assume that for $i=1,2$,
    \begin{equation}\label{AssumpBound3}
    \|\varphi_i\|_{X(T)}\leq M_1, \abs{\varphi_i}_*\geq m_1>0, d\paren{\varphi_i}\geq m_2>0 .
    \end{equation}
    Then, it holds that for any $\delta\in(0,\gamma)$
\begin{equation*}
    t^\frac{\varepsilon}{3}\norm{\mathcal{N}(\varphi_1,\sigma_1)-\mathcal{N}(\varphi_2,\sigma_2)}_{C^{1,\gamma+\varepsilon}}\leq C M_2^3 t^{\frac{\gamma-\delta}{3}-1}\norm{\varphi_1-\varphi_2}_{X(T)},
\end{equation*}
where $C$ depends on $\gamma,\varepsilon,\delta, M_1, m_1$ and $M_2=\max \braces{\sup_{t\in [0,T]} C_{\varphi_1 \paren{t}},\sup_{t\in [0,T]} C_{\varphi_2 \paren{t}},1}$
\end{prop}
\begin{proof}

We start with $\mc{N}_1$ \eqref{Nonlinear_split}:
\begin{equation*}
\begin{aligned}
\mathcal{N}_1(\varphi_1,\sigma_1)&-\mathcal{N}_1(\varphi_2,\sigma_2)=\frac{1}{4}\mc{H}\Big(\frac{1}{2}(\p_s\varphi_1-\p_s\varphi_2)\big((1+\p_s\varphi_1)^2+(1+\p_s\varphi_2)^2\\
&\quad+(1+\p_s\varphi_1)(1+\p_s\varphi_2)\big)-\frac{1}{4}\mc{H}\Big((\sigma_1-\sigma_2)(1+\p_s\varphi_1)+\sigma_2\paren{\p_s\varphi_1-\p_s\varphi_2}\Big).
\end{aligned}
\end{equation*}
Then,
\begin{equation*}
\begin{aligned}
\|&\mathcal{N}_1(\varphi_1,\sigma_1)-\mathcal{N}_1(\varphi_2,\sigma_2)\|_{C^{1,\gamma+\varepsilon}}\\
&\leq C\Big( \|\p_s\varphi_1-\p_s\varphi_2\|_{C^{1,\gamma+\varepsilon}}\big(1+\|\p_s\varphi_1\|_{C^{0}}^2+\|\p_s\varphi_2\|_{C^{0}}^2\big)\\
&\quad+\|\p_s\varphi_1-\p_s\varphi_2\|_{C^0}\big(1+\|\p_s\varphi_1\|_{C^{1,\gamma+\varepsilon}}+\|\p_s\varphi_2\|_{C^{1,\gamma+\varepsilon}}\big)\big(1+\|\p_s\varphi_1\|_{C^0}+\|\p_s\varphi_2\|_{C^0}\big)\\
&\quad+\|\sigma_1-\sigma_2\|_{C^{1,\gamma+\varepsilon}}(1+\|\p_s\varphi\|_{C^{0}})+\|\sigma_1-\sigma_2\|_{C^{0}}(1+\|\p_s\varphi_1\|_{C^{1,\gamma+\varepsilon}})\\
&\quad+\|\sigma_2\|_{C^{1,\gamma+\varepsilon}}\|\p_s\varphi_1-\p_s\varphi_2\|_{C^{0}}+\|\sigma_2\|_{C^{0}}\|\p_s\varphi_1-\p_s\varphi_2\|_{C^{1,\gamma+\varepsilon}}\Big).
\end{aligned}
\end{equation*}
Using the interpolation inequalities \eqref{interplationbddeq03} and \eqref{AssumpBound3}, it simplifies to

\begin{equation*}
\begin{split}
        &t^{\frac{\varepsilon}{3}}\norm{\mathcal{N}_1(\varphi_1,\sigma_1)-\mathcal{N}_1(\varphi_2,\sigma_2)}_{C^{1,\gamma+\varepsilon}}\\
    \leq&C \left(t^{-\frac{1}{3}}\norm{\varphi_1-\varphi_2}_{X\paren{T}}+t^{\frac{\varepsilon}{3}}\norm{\sigma_1-\sigma_2}_{C^{1,\gamma+\varepsilon}}+t^{-\frac{1}{3}}\norm{\sigma_1-\sigma_2}_{C^{0}}\right.\\
        &\quad\left.+t^{\frac{\varepsilon}{3}}\norm{\sigma_2}_{C^{1,\gamma+\varepsilon}}\norm{\varphi_1-\varphi_2}_{X\paren{T}}+t^{-\frac{1}{3}}\norm{\sigma_2}_{C^{0}}\norm{\varphi_1-\varphi_2}_{X\paren{T}}\right).
\end{split}
\end{equation*}
We now plug in the tension estimates from Proposition \ref{TensionProp2}. First, take Proposition \ref{TensionProp2} and \eqref{interplationbddeq03},
\begin{equation*}
\begin{aligned}
    \|\sigma_i\|_{C^{1,\alpha}}&\leq C M_2\big(1+t^{-\frac{2-\paren{\gamma-\delta}}{3}}\big),
\end{aligned}
\end{equation*}
\begin{equation*}
\begin{aligned}
    \|\sigma_1(\varphi_1)-\sigma_2(\varphi_2)\|_{C^{1,\alpha}}\leq 
     C M_2^2\|\varphi_1-\varphi_2\|_{X(T)}\big(1+t^{-\frac{2-\paren{\gamma-\delta}}{3}}\big),
\end{aligned}
\end{equation*}
    \begin{equation*}
        \begin{aligned}
              \|\sigma_i\|_{C^{2,\alpha}}&\leq C M_2\big(t^{-\frac{1-\paren{\gamma-\delta}}{3}}+t^{-\frac{3-2\paren{\gamma-\delta}}{3}}+t^{-\frac{3-\paren{\gamma-\delta}}{3}}\big),
        \end{aligned}
    \end{equation*}
    \begin{equation*}
\begin{aligned}
    \|\sigma_1&(\varphi_1)-\sigma_2(\varphi_2)\|_{C^{2,\alpha}}\leq C M_2^3\|\varphi_1-\varphi_2\|_{X(T)}\big(1+t^{-\frac{1-\paren{\gamma-\delta}}{3}}+t^{-\frac{2-\paren{\gamma-\delta}}{3}}+t^{-\frac{3-\paren{\gamma-\delta}}{3}}\big).
\end{aligned}
\end{equation*}
Substituting the estimates for the tension with $M_2>1$, we conclude that
\begin{equation*}
\begin{split}
    t^{\frac{\varepsilon}{3}}\norm{\mathcal{N}_1(\varphi_1,\sigma_1)-\mathcal{N}_1(\varphi_2,\sigma_2)}_{C^{1,\gamma+\varepsilon}}
    \leq& C M_2^3\paren{t^{\frac{\varepsilon}{3}}+t^{-\frac{1}{3}}}\paren{1+t^{-\frac{2-\paren{\gamma-\delta}}{3}}}\norm{\varphi_1-\varphi_2}_{X\paren{T}}\\
    \leq&C M_2^3 t^{\frac{\gamma-\delta}{3}-1}\norm{\varphi_1-\varphi_2}_{X\paren{T}}.
\end{split}
\end{equation*}
Next, we proceed with $\mc{N}_2$. We use the notation in \eqref{NormalTanForce},
\begin{equation*}
    \begin{aligned}
        \mathcal{N}_2(\varphi_1,\sigma_1)-\mathcal{N}_2(\varphi_2,\sigma_2)&=-\frac14(\bm{n}_1-\bm{n}_2)\cdot \big([\mc{H},\bm{n}_1]F_{1,n}+[\mc{H},\bm{\tau}_1]F_{1,\tau}\big)\\
        &\quad-\frac14\bm{n}_2\cdot\big([\mc{H},\bm{n}_1-\bm{n}_2]F_{1,n}+[\mc{H},\bm{\tau}_1-\bm{\tau}_2]F_{1,\tau}\big)\\
        &\quad-\frac14\bm{n}_2\cdot[\mc{H},\bm{n}_2]\big(F_{1,n}-F_{2,n}+F_{1,\tau}-F_{2,\tau}\big),
    \end{aligned}
\end{equation*}
and compute bounds for its norm
\begin{equation*}
    \begin{aligned}
        \|\mathcal{N}_2(\varphi_1,\sigma_1)-&\mathcal{N}_2(\varphi_2,\sigma_2)\|_{C^{1,\gamma+\varepsilon}}\\
        &\leq C\|\varphi_1-\varphi_2\|_{C^{1,\gamma+\varepsilon}} \big( \|[\mc{H},\bm{n}_1]F_{1,n}\|_{C^{1,\gamma+\varepsilon}}+\|[\mc{H},\bm{\tau}_1]F_{1,\tau}\|_{C^{1,\gamma+\varepsilon}}\big)\\
        &\quad+C(1+\|\varphi_2\|_{C^{1,\gamma+\varepsilon}})\big(\|[\mc{H},\bm{n}_1-\bm{n}_2]F_{1,n}\|_{C^{1,\gamma+\varepsilon}}+\|[\mc{H},\bm{\tau}_1-\bm{\tau}_2]F_{1,\tau}\|_{C^{1,\gamma+\varepsilon}}\big)\\
        &\quad+C(1+\|\varphi_2\|_{C^{1,\gamma+\varepsilon}})\|[\mc{H},\bm{n}_2]\big(F_{1,n}-F_{2,n}+F_{1,\tau}-F_{2,\tau}\|_{C^{1,\gamma+\varepsilon}}\big).
    \end{aligned}
\end{equation*}
We apply the commutator estimate \eqref{comm1} to obtain
\begin{equation}\label{N2_cont_split}
    \begin{aligned}
        \|\mathcal{N}_2(\varphi_1,\sigma_1)-\mathcal{N}_2(\varphi_2,\sigma_2)\|_{C^{1,\gamma+\varepsilon}}\leq C\big(I_1+I_2+I_3\big),
\end{aligned}
\end{equation}
where
\begin{equation*}
    \begin{aligned}
        I_1&\!=\!\|\varphi_1\!-\!\varphi_2\|_{C^{1,\gamma+\varepsilon}} \Big((1\!+\!\|\varphi_1\|_{C^{2}})\big(\|F_{1,n}\|_{C^{0}}\!+\!\|F_{1,\tau}\|_{C^{0}}\big)\!+\!(1\!+\!\|\varphi_1\|_{C^{1}})\big(\|F_{1,n}\|_{C^{1}}\!+\!\|F_{1,\tau}\|_{C^{1}}\big)\Big),\\
        I_2&\!=\!(1+\|\varphi_2\|_{C^{1,\gamma+\varepsilon}})\!\Big( \|\varphi_1\!-\!\varphi_2\|_{C^{2}}\big(\|F_{1,n}\|_{C^{0}}\!+\!\|F_{1,\tau}\|_{C^{0}}\big)\!+\!\|\varphi_1\!-\!\varphi_2\|_{C^{1}}\big(\|F_{1,n}\|_{C^{1}}\!+\!\|F_{1,\tau}\|_{C^{1}}\big)\!\Big),\\
        I_3&\!=\!(1+\|\varphi_2\|_{C^{1,\gamma+\varepsilon}})\Big( (1+\|\varphi_2\|_{C^{2}})\big(\|F_{1,n}-F_{2,n}\|_{C^{0}}+\|F_{1,\tau}-F_{2,\tau}\|_{C^{0}}\big)\\
        &\hspace{2cm}\qquad+(1+\|\varphi_2)\|_{C^{1}}\big(\|F_{1,n}-F_{2,n}\|_{C^{1}}+\|F_{1,\tau}-F_{2,\tau}\|_{C^{1}}\big)\Big).
    \end{aligned}
\end{equation*}
Next, we substitute the bounds for $F_{i,n}, F_{i,\tau}$, $i=1,2$, from Proposition \ref{force_est}, interpolate the norms and use condition \eqref{AssumpBound3} to find that
\begin{equation*}
    \begin{aligned}
        I_1&\leq C\|\varphi_1-\varphi_2\|_{C^{1,\gamma+\varepsilon}} \\
        &\quad\times\Big( (1+\|\varphi_1\|_{C^{2}})\big(1+\|\varphi_1\|_{C^{3}}+\|\sigma_1\|_{C^{1}}\big)+\|\sigma_1\|_{C^2}+1+\|\varphi_1\|_{C^{4}}+\paren{1+\|\sigma_1\|_{C^{0}}}\|\varphi_1\|_{C^{2}}\Big),
    \end{aligned}
\end{equation*}
\begin{equation*}
    \begin{aligned}
        I_2&\leq C(1+\|\varphi_2\|_{C^{1,\gamma+\varepsilon}})\Big( \|\varphi_1\!-\!\varphi_2\|_{C^{2}}\big(1+\|\varphi_1\|_{C^{3}}+\|\sigma_1\|_{C^{1}}\big)\\
        &\hspace{2cm}+\|\varphi_1-\varphi_2\|_{C^1}\big(\|\sigma_1\|_{C^2}+1+\|\varphi_1\|_{C^{4}}+\paren{1+\|\sigma_1\|_{C^{0}}}\|\varphi_1\|_{C^{2}}\big)\Big),
    \end{aligned}
\end{equation*}
\begin{equation*}
    \begin{aligned}
        I_3&\leq C(1+\|\varphi_2\|_{C^{1,\gamma+\varepsilon}})\Big(  (1+\|\varphi_2\|_{C^{2}})\big(\|\varphi_1-\varphi_2\|_{C^{3}}+\|\varphi_1-\varphi_2\|_{C^1}+\|\sigma_1-\sigma_2\|_{C^0}\\
            &\qquad\quad+\|\sigma_2\|_{C^0}\|\varphi_1-\varphi_2\|_{C^1}+\|\sigma_1-\sigma_2\|_{C^{1}}\big)\\
        &\hspace{2cm}+\|\varphi_1\!-\!\varphi_2\|_{C^{4}}\!+\!\big(1\!+\!\|\sigma\|_{C^0}\big)\|\varphi_1\!-\!\varphi_2\|_{C^{2}}\\
            &\qquad+\big(1+\|\varphi_1\|_{C^2}+\|\varphi_2\|_{C^2}\big)\|\varphi_1-\varphi_2\|_{C^{1}}+\|\sigma_1-\sigma_2\|_{C^1}\\
            &\qquad+\|\sigma_1-\sigma_2\|_{C^0}(1+\|\varphi_1\|_{C^2})+\|\sigma_2\|_{C^1}\|\varphi_1-\varphi_2\|_{C^1}+\|\sigma_1-\sigma_2\|_{C^{2}}\Big).
    \end{aligned}
\end{equation*}
Proceeding with $I_1$ by the above tension estimates and \eqref{interplationbddeq04}, we obtain
\begin{equation*}
\begin{aligned}
    t^{\frac{\varepsilon}{3}}I_1&\leq  C M_2^3\|\varphi_1-\varphi_2\|_{X(T)}\Big(t^{-\frac{1-\paren{\gamma-\delta}}{3}}+t^{-\frac{3-2\paren{\gamma-\delta}}{3}}+t^{-\frac{3-\paren{\gamma-\delta}}{3}}\Big).
\end{aligned}
\end{equation*}
Similarly, the estimate for $I_2$ becomes
\begin{equation*}
\begin{aligned}
    t^{\frac{\varepsilon}{3}}I_2&\leq  C M_2^3\|\varphi_1-\varphi_2\|_{X(T)}\Big(t^{-\frac{1-\paren{\gamma-\delta}}{3}}+t^{-\frac{3-2\paren{\gamma-\delta}}{3}}+t^{-\frac{3-\paren{\gamma-\delta}}{3}}\Big).
\end{aligned}
\end{equation*}
Finally, the bound for $I_3$ is given by
\begin{equation*}
\begin{aligned}
    t^{\frac{\varepsilon}{3}}I_3&\leq  C M_2^3\|\varphi_1-\varphi_2\|_{X(T)}\Big(1+t^{-\frac{1-\paren{\gamma-\delta}}{3}}+t^{-\frac{2-\paren{\gamma-\delta}}{3}}+t^{-\frac{3-2\paren{\gamma-\delta}}{3}}+t^{-\frac{3-\paren{\gamma-\delta}}{3}}\Big).
\end{aligned}
\end{equation*}
Joining the bounds for $I_1$, $I_2$, and $I_3$ back into \eqref{N2_cont_split}, we conclude
\begin{equation*}
\begin{aligned}
t^{\frac{\varepsilon}{3}}\|&\mathcal{N}_2(\varphi_1,\sigma_1)-\mathcal{N}_2(\varphi_2,\sigma_2)\|_{C^{1,\gamma+\varepsilon}}\leq C M_2^3 t^{-\frac{3-\paren{\gamma-\delta}}{3}}\|\varphi_1-\varphi_2\|_{X(T)}.
\end{aligned}
\end{equation*}
Finally, the contraction estimate for $\mc{N}_3$ is obtained in the same way as the one for $\mc{N}_2$, using Proposition \ref{RestimateProp} instead of Lemma \ref{CommutatorLem} (given that $|\bm{X}_i|_*>m_1$, $i=1,2$).

\end{proof}

\subsection{Fixed Point}\label{fixedpoint}

We proceed to show the existence of a unique solution to our problem \eqref{theta_eqn}-\eqref{tension_eqn}. We do it by showing the existence of a unique fixed point in $X(T)$ \eqref{space_fixed} for the map 
\begin{equation}\label{fixed-point}
    \begin{aligned}
        \Theta(\varphi,t;\varphi_0)=e^{-\frac14|\nabla|^3t}\varphi_0+\int_0^te^{-\frac14|\nabla|^3(t-\tau)}\mathcal{N}(\varphi,\sigma(\varphi))(\tau)d\tau. 
    \end{aligned}
\end{equation}
We recall the definition
\begin{equation*}
    \begin{aligned}
        d(\varphi)=\min_{a\in\mathbb{S}^1}\|\varphi(s)-a\|_{C^0}.
    \end{aligned}
\end{equation*}
Without loss of generality, we assume that the initial interface is not a circle, since there are no dynamics in such case. Hence
\begin{equation*}
    \begin{aligned}
        d(\varphi_0)>0.
    \end{aligned}
\end{equation*}
We will use the Banach fixed point theorem for $\Theta$ on a suitable closed subset of $X(T)$ \eqref{space_fixed}, given by
\begin{equation*}
    \begin{aligned}
        O_T&\eqdef \{\varphi\in X(T):\|\varphi-e^{-\frac14|\nabla|^3t}\varphi_0\|_{X(T)}\leq \frac12\min\{|\varphi_0|_*,d(\varphi_0)\}\}\subset X(T),\\
        \mc{D}_T&\eqdef\set{\vartheta \in C^{1,\gamma}\paren{\mbs}}{\vartheta(s)=\varphi\paren{s,t} \mbox{ for some } \varphi\in O_T \mbox{ and } t\in [0,T]}.
    \end{aligned}
\end{equation*}
Notice that for $\varphi\in O_T$, $t\in[0,T]$, it holds that,
\begin{equation}\label{aux_arc}
    \begin{aligned}
        \|\varphi\|_{X(T)}&\leq \|e^{-\frac14|\nabla|^3t}\varphi_0\|_{X(T)}+\frac12|\varphi_0|_*\\
        &\leq C\|\varphi_0\|_{C^{1,\gamma}}+\frac12|\varphi_0|_*,\\
        \norm{\varphi\paren{t}-\varphi_0}_{C^{1,\gamma}}&\leq \|\varphi(t)-e^{-\frac14|\nabla|^3t}\varphi_0\|_{C^{1,\gamma}}+\|e^{-\frac14|\nabla|^3t}\varphi_0-\varphi_0\|_{C^{1,\gamma}},\\
        |\varphi(t)|_*&\geq |\varphi_0|_*-\|\varphi(t)-\varphi_0\|_{C^{1,\gamma}}\\
        &\geq |\varphi_0|_*-\|\varphi(t)-e^{-\frac14|\nabla|^3t}\varphi_0\|_{C^{1,\gamma}}-\|e^{-\frac14|\nabla|^3t}\varphi_0-\varphi_0\|_{C^{1,\gamma}},\\
        d(\varphi(t))&\geq \frac12 d(\varphi_0)-\|e^{-\frac14|\nabla|^3t}\varphi_0-\varphi_0\|_{C^{0}}.
    \end{aligned}
\end{equation}

\begin{thm}\label{theo_fixed}
     Let $\varphi_0\in h^{1,\gamma}(\mathbb{S}^1)$, $\gamma\in(0,1)$, $|\varphi_0|_*>0$. There exists $T>0$ such that $\Theta(\varphi,t;\varphi_0)$ 
forms a contraction on $O_T$. 
\end{thm}

\begin{proof}

We first show that, fixed $\varphi_0\in h^{1,\gamma}(\mathbb{S}^1)$, $\Theta$ maps $O_{T_1}$ to itself for some $T_1>0$.
We choose $T_1\leq 1$ small enough so that
\begin{equation}\label{exptheta0}
    \|e^{-\frac14|\nabla|^3t}\varphi_0-\varphi_0\|_{C^{1,\gamma}}\leq \frac14 \min\{|\varphi_0|_*,d(\varphi_0)\}.
\end{equation}
Therefore, for any $\varphi\in O_{T_1}$ and $t\leq T_1$,
\begin{equation}\label{cond_aux}
    \begin{aligned}
        \|\varphi\|_{X(T)}\leq C\|\varphi_0\|_{C^{1,\gamma}}+\frac12|\varphi_0|_*, \quad|\varphi(t)|_*\geq \frac{1}{4}|\varphi_0|_*, \quad
        d(\varphi(t))\geq \frac{1}{4}d(\varphi_0),         
    \end{aligned}
\end{equation}
and we may set 
\begin{align*}
    M_1=C\|\varphi_0\|_{C^{1,\gamma}}+\frac12|\varphi_0|_*, m_1=\frac{1}{4}|\varphi_0|_*.
\end{align*}
By Proposition \ref{Tensionbdn}, $C_\varphi$ is continuous in $\mc{D}_{T_1}$ for $C^{1}(\mathbb{S}^1)$ norm. 
By \eqref{cond_aux} and the fact that $C^{1,\gamma}(\mathbb{S}^1)$ is compactly embedded in $C^{1}(\mathbb{S}^1)$, $\mc{D}_{T_1}$ is precompact in $C^{1}(\mathbb{S}^1)$ and the unit circle $\paren{\braces{d(\varphi)=0}}$ is not the limit point of $\mc{D}_{T_1}$, so
\begin{equation*}
\begin{aligned}
    M_3:=  \max \Big\{\sup_{\varphi\in\mc{D}_{T_1}} C_\varphi,1\Big\}<\infty,
\end{aligned}
\end{equation*}
and for the following estimates, we may replace $M_2$ in Proposition \ref{nonli_est} and \ref{Contract_est} with $M_3$.
Then, for any $0<T_2\leq T_1$,
\begin{equation*}
    \begin{aligned}
        \|\Theta(\varphi,t;\varphi_0)-e^{-\frac14|\nabla|^3t}\varphi_0\|_{X(T_2)}&\leq\Big|\Big|\int_0^te^{-\frac14|\nabla|^3(t-\tau)}\mathcal{N}(\varphi,\sigma(\varphi))(\tau)d\tau\Big|\Big|_{X(T_2)}. 
    \end{aligned}
\end{equation*}
Lemma \ref{SemigroupLem} and Proposition \ref{nonli_est} with $\delta=\frac{\gamma}{2}$ give that
\begin{equation*}
    \begin{aligned}
        &\norm{\int_0^t e^{-\frac14|\nabla|^3 (t-\tau)}\mathcal{N}(\varphi,\sigma)(\tau)d\tau}_{C^{4,\gamma}}
        \leq C \int_0^t\frac{1}{(t-\tau)^{1-\frac{\varepsilon}{3}}}\norm{ \mathcal{N}(\varphi,\sigma)(\tau)}_{C^{1,\gamma+\varepsilon}}d\tau\\
        &\leq C \int_0^t\frac{1}{(t-\tau)^{1-\frac{\varepsilon}{3}}\tau^{1+\frac{\varepsilon}{3}-\frac{\gamma}{6}}}d\tau
        \leq C M_3\int_0^t\paren{\frac{1}{t^{1-\frac{\varepsilon}{3}}\tau^{1+\frac{\varepsilon}{3}-\frac{\gamma}{6}}}+\frac{1}{(t-\tau)^{1-\frac{\varepsilon}{3}}t^{1+\frac{\varepsilon}{3}-\frac{\gamma}{6}}}}d\tau\\
        &\leq C M_3 t^{\frac{\gamma}{6}-1},
    \end{aligned}
\end{equation*}
and also
\begin{equation*}
    \begin{aligned}
     \Big|\Big|\int_0^te^{-\frac14|\nabla|^3 (t-\tau)}\mathcal{N}(\varphi,\sigma)(\tau)d\tau\Big|\Big|_{C^{1,\gamma}}&\leq C M_3\int_0^t \|\mathcal{N}(\varphi,\sigma)\|_{C^{1,\gamma}}(\tau)d\tau
     \leq C M_3\int_0^t \tau^{\frac{\gamma}{3}-1}d\tau
     \leq C M_3 t^{\frac{\gamma}{6}},
    \end{aligned}
\end{equation*}
thus
\begin{equation*}
    \begin{aligned}
     \|\Theta(\varphi,t;\varphi_0)-e^{-\frac14|\nabla|^3t}\varphi_0\|_{X(T_2)}\leq C\Big|\Big|\int_0^te^{-\frac14|\nabla|^3 (t-\tau)}\mathcal{N}(\varphi,\sigma)(\tau)d\tau\Big|\Big|_{X(T_2)}&\leq C M_3 T_2^{\frac{\gamma}{6}}.
    \end{aligned}
\end{equation*}
Where $C$ depends on  $|\varphi_0|_*$, $\|\varphi_0\|_{C^{1,\gamma}}$.
Therefore, we can take $T_2>0$ small enough so that
\begin{equation*}
    \|\varphi-e^{-\frac14|\nabla|^3t}\varphi_0\|_{X(T)}\leq \frac12\min\{|\varphi_0|_*,d(\varphi_0)\}.
\end{equation*}
Next, we show that $\Theta$ forms a contraction on $O_{T_3}$ for some $T_3>0$.
Let $\varphi_1, \varphi_2\in O_{T_3}$. Then, 
\begin{equation*}
    \begin{aligned}
        \|\Theta(\varphi_1,t;\varphi_{1,0})-&\Theta(\varphi_2,t;\varphi_{2,0})\|_{X(T_3)}\\
        &\leq\Big|\Big|\int_0^te^{-\frac14|\nabla|^3(t-\tau)}\big(\mathcal{N}(\varphi_1,\sigma(\varphi_1))(\tau)-\mathcal{N}(\varphi_2,\sigma(\varphi_2))(\tau)\big)d\tau\Big|\Big|_{X(T_3)}.
    \end{aligned}
\end{equation*}
We use Lemma \ref{SemigroupLem} and Proposition \ref{Contract_est} with $\delta=\frac{\gamma}{2}$ to obtain
\begin{equation*}
    \begin{aligned}
        \|\Theta(\varphi_1,t;\varphi_{1,0})&-\Theta(\varphi_2,t;\varphi_{2,0})\|_{X(T_3)}\leq C M_2^3 T_3^{\frac{\gamma}{3}}\|\varphi_1-\varphi_2\|_{X(T_3)},
    \end{aligned}
\end{equation*}
where $C$ depends on $\|\varphi_{1,0}\|_{C^1}$,  $\|\varphi_{2,0}\|_{C^1}$, $|\varphi_{1,0}|_*$, $|\varphi_{2,0}|_*$.
Hence, we can choose $T_3>0$ small enough so that
\begin{equation*}
    \begin{aligned}
        \|\Theta(\varphi_1,t;\varphi_{1,0})&-\Theta(\varphi_2,t;\varphi_{2,0})\|_{X(T_3)}\leq \frac12\|\varphi_1-\varphi_2\|_{X(T_3)},
    \end{aligned}
\end{equation*}
and the result follows by taking $T=\min\{T_2,T_3\}$.

\end{proof}

\begin{remark}\label{length_of_T}
Note that
\begin{equation*}
\norm{e^{-\frac14|\nabla|^3t}\varphi_0-\varphi_0}_{C^{1,\gamma/2}}\leq C_0t^{\gamma/6}\norm{e^{-\frac14|\nabla|^3t}\varphi_0-\varphi_0}_{C^{1,\gamma}}
\leq C_1t^{\gamma/6}\norm{\varphi_0}_{C^{1,\gamma}},
\end{equation*}
where the constants $C_0$ and $C_1$ only depend on $\gamma$.
This implies that, by performing the above fixed point argument by replacing the H\"older exponent $\gamma$ with $\gamma/2$,
we can take $T_1$ in the above proof
to depend only on $\norm{\varphi_0}_{C^{1,\gamma}}$
(see equation \eqref{exptheta0}). 
This implies that the existence time $T$ can be taken to depend only on $d(\varphi_0)$, $\abs{\varphi_0}_*$ and the smoothness of $\varphi_0$. 
\end{remark}

\subsection{Parabolic smoothing}\label{smoothing}

Given $h_0\in h^{1,\gamma}(\mathbb{S}^1)$ and the regularity of functions in $X(T)$, defined in \eqref{space_fixed}, it is easy to see that the fixed point constructed in Theorem \ref{theo_fixed} is in fact a strong solution in the sense of Definition \ref{def_sol}. We moreover show that this solution is smooth.

\begin{prop}[Smoothing]\label{smoothing}
Let $\varphi$ be the solution to the Inextensible Interface problem in $[0,T]$ with initial data $\varphi_0\in h^{1,\gamma}(\mathbb{S}^1)$, $\gamma\in(0,1)$. Then,  for any $n\in\mathbb{N}$ and $\alpha\in(0,1)$, it holds that $\varphi\in C((0,T];C^{n,\alpha}(\mathbb{S}^1))$.
\end{prop}

\begin{proof}
Since $\varphi\in X(T)$, we have that for $t\in(0,T]$, $\|\varphi(t)\|_{C^{4,\gamma}}\leq C(t)$. 
Given $t_0\in(0,T]$, we write
\begin{equation}\label{duhamel_smo}
    \begin{aligned}
        \varphi(t)=e^{-\frac14|\nabla|^3(t-t_0)}\varphi(t_0)+\int_{t_0}^te^{-\frac14|\nabla|^3(t-\tau)}\mathcal{N}(\varphi,\sigma(\varphi))(\tau)d\tau. 
    \end{aligned}
\end{equation}
We first show that $\varphi(t)\in C^{4,\alpha}(\mathbb{S}^1)$ for any $\alpha\in(0,1)$. 
Applying Lemma \ref{SemigroupLem} in \eqref{duhamel_smo} with $\varepsilon\in(0,1-\alpha)$, we have
\begin{equation*}
    \begin{aligned}
        \|\varphi(t)\|_{C^{4,\alpha}}&\leq C\frac{\|\varphi(t_0)\|_{C^{1,\alpha}}}{t-t_0}+\int_{t_0}^t\frac{C}{(t-\tau)^{1-\frac{\varepsilon}{3}}}\|\mathcal{N}(\varphi,\sigma(\varphi))(\tau)\|_{C^{1,\alpha+\varepsilon}}d\tau.
    \end{aligned}
\end{equation*}
If $\gamma\in[2/3,1)$, then Proposition \ref{nonli_est} with $\delta=\frac{\gamma}{2}$ gives that for any $\beta \in(\gamma,1)$,
\begin{equation*}
    \|\mc{N}(\varphi,\sigma)(t)\|_{C^{1,\beta}}\leq C M_2^2 t^{\frac{\frac{3}{2}\gamma-\beta}{3}-1},
\end{equation*}
thus we conclude that $\varphi(t)\in C^{4,\alpha}(\mathbb{S}^1)$ for any $t>t_0$. Otherwise, when $\gamma\in(0,2/3)$, the same reasoning with $\beta\in(\gamma,\frac{3}{2}\gamma)$ gives that $\varphi(t)\in C^{4,\alpha}(\mathbb{S}^1)$ for $\alpha\in(0,\frac{3}{2}\gamma)$, $t>t_0$. Doing the same procedure $k$ times, until $\paren{\frac{3}{2}}^k\gamma\geq1$, gives the result for the whole range $\alpha\in(0,1)$. Since $t_0>0$ was arbitrary, we conclude that $\varphi(t)\in C^{4,\alpha}(\mathbb{S}^1)$ for $t>0$, $\alpha\in(0,1)$.

Next, we show that $\varphi(t)\in C^{4+k,\alpha}(\mathbb{S}^1)$ for $t>0$, $\alpha\in(0,1)$, $0\leq k\in\mathbb{Z}$. We proceed by induction: assume that $\varphi(t)\in C^{4+k,\alpha}(\mathbb{S}^1)$ for $t>0$, $\alpha\in(0,1)$. In particular, from estimate \eqref{sigma_kest}, we then have that $\sigma(t)\in C^{2+k,\alpha}(\mathbb{S}^1)$. Arguing as above, we have that
\begin{equation*}
    \begin{aligned}
        \|\varphi(t)\|_{C^{5+k,\alpha}}&\leq C\frac{\|\varphi(t_0)\|_{C^{2+k,\alpha}}}{t-t_0}+\int_{t_0}^t\frac{C}{(t-\tau)^{1-\frac{\varepsilon}{3}}}\|\mathcal{N}(\varphi,\sigma(\varphi))(\tau)\|_{C^{2+k,\alpha+\varepsilon}}d\tau.
    \end{aligned}
\end{equation*}
However, an analogous bound to the previous one on $\mc{N}(\varphi,\sigma(\varphi))$  is no longer valid, as it would involve the $C^{5+k}(\mathbb{S}^1)$ norm of $\varphi(t)$. We recall the splitting \eqref{Nonlinear_split} and perform $C^{2+k,\beta}(\mathbb{S}^1)$ estimates, $\beta=\alpha+\varepsilon\in(0,1)$. A straightforward bound for the first term follows: 
\begin{equation*}
\begin{aligned}
\|\mathcal{N}_1(\varphi,\sigma)(t)\|_{C^{2+k,\beta}}&\leq\Big(\|\varphi(t)\|_{C^{3+k,\beta}}(1+\|\varphi(t)\|_{C^{1}})^2+\|\varphi(t)\|_{C^{3+k,\beta}}\|\sigma(t)\|_{C^{0}}\\
&\qquad+(1+\|\varphi(t)\|_{C^{1}})\|\sigma(t)\|_{C^{2+k,\beta}}\Big)\\
&\leq C(t).
\end{aligned}
\end{equation*}
We then use Lemma \ref{CommutatorLem}, and in particular estimate \eqref{comm2},  to obtain that
\begin{equation*}
\begin{aligned}
\|\mathcal{N}_2(\varphi,\sigma)(t)\|_{C^{2+k,\beta}}&\leq C(t)\Big(1+\|[\mc{H},\bm{n}]\p_s^3\varphi\|_{C^{2+k,\beta}}+\|[\mc{H},\bm{n}]\p_s\sigma\|_{C^{2+k,\beta}}\Big)\\
&\leq C(t)\Big(1+\|\varphi(t)\|_{C^{4+k}}+\|\sigma(t)\|_{C^{2+k}}\Big)\\
&\leq C(t).
\end{aligned}
\end{equation*}
The bound for $\mc{N}_3 $ follows exactly as the one for $\mc{N}_2$ but using instead the estimate \eqref{Rest2} from Proposition \ref{RestimateProp}. Therefore, we have shown that $\varphi(t)\in C^{5+k,\alpha}(\mathbb{S}^1)$ for any $\alpha\in(0,1)$ and $t>t_0$. Since $t_0>0$ was arbitrary, we conclude the induction argument.
    
\end{proof}
The regularity in time then follows from the regularity in space and repeated use of \eqref{duhamel_smo}, concluding that $\varphi\in C^\infty((0,T]\times\mathbb{S}^1)$.

\section{Numerical Simulations}\label{sec5}

We recall that equations \eqref{theta_eqn}, \eqref{tension_eqn} and \eqref{NormalTanForce} can be written as
\begin{equation}\label{eqNum}
\begin{split}
\p_t \varphi
=&\frac{1}{4}\mc{H}F_n(\varphi, \sigma)
-\frac{1}{4}\bm{n}(s)\cdot[\mc{H},\bm{n}]F_n(\varphi, \sigma)-\frac{1}{4}\bm{n}(s)\cdot[\mc{H},\bm{\tau}]F_\tau(\sigma)\\
&+\bm{n}(s)\cdot\PPD{s}{\mc{R}^{\paren{-1}}}\paren{F_n(\varphi, \sigma) \bm{n}(s)+F_\tau(\sigma)\bm{\tau}(s)}\\
=&\frac{1}{4}\mc{H}\p_s^3\varphi+\mc{N}\paren{\varphi,\sigma},
\end{split}
\end{equation}
where
\begin{equation*}
\begin{aligned}
    F_n(\varphi, \sigma)=\p_s^3\varphi +\frac{1}{2}(\p_s\varphi+1)^3-\sigma\paren{\p_s\varphi+1},\quad F_\tau(\sigma)=\p_s\sigma,
\end{aligned}
\end{equation*}
and
\begin{align*}
\begin{split}
        &\quad\bm{\tau}\cdot\paren{-\frac{1}{4}\mc{H}+\PPD{s}{\mc{R}^{\paren{-1}}}}\paren{\p_s\paren{\sigma \bm{\tau}}}\\
        =&-\bm{\tau}\cdot\paren{-\frac{1}{4}\mc{H}+\PPD{s}{\mc{R}^{\paren{-1}}}}\paren{\p_s{\paren{\p_{s}^{2}{\varphi}\bm{n}-\frac{1}{2}\paren{\p_s \varphi+1}^2\bm{\tau}}}}.
\end{split}
\end{align*}   
In addition, the equations for $\bm{X}$ are
\begin{align*}
    \bm{X}\paren{s,t}           =&\bm{X}_c\paren{t}+\int \bm{\tau}\paren{s,t}ds,\\
    \PPD{t}{\bm{X}_c\paren{t}}   =&\frac{1}{2\pi}\int_\mbs \mc{R}^{\paren{-1}}\paren{F_n(\varphi, \sigma) \bm{n}(s)+F_\tau(\sigma)\bm{\tau}(s)}ds.
\end{align*}
Now, we can give a description of the numerical scheme.
\subsection{Spatial Discretizations}
We first discretize the spatial variable $s\in \mbs$ with $N$ points s.t. $s_j=jh$, where $j=0,1,2,\cdots, N-1$ and $h=2\pi/N$, and a function $f\paren{s}$ is discretized as $f_j=f\paren{s_j}$.
Then, we define the discrete Fourier Transform $\mc{F}_h$ from $f_j$ into $\hat{f}_k$ by
\begin{equation*}
    \hat{f}_k=\paren{\mc{F}_h[f_j]}_k:=\sum_{j=0}^{N-1}f_je^{\inum ks_j},\quad k=-\frac{N}{2}+1,\cdots, -1, 0, 1,\cdots, \frac{N}{2},
\end{equation*}
and the inverse discrete Fourier Transform $\mc{F}_h^{-1}$ by
\begin{equation*}
    f_j=\paren{\mc{F}_h^{-1}[\hat{f}_k]}_j:=\sum_{k=-\frac{N}{2}+1}^{\frac{N}{2}}\hat{f}_k e^{-\inum ks_j}.
\end{equation*}
In these equations, the main operators are derivatives $\p_s^n$, Hilbert transform $\mc{H}$, and the singular integral operator $\mc{R}$.
Since the Fourier multipliers of $\p_s^k$ and $\mc{H}$ on $\mbs$ are $\paren{\inum k}^n$ and  $-\inum~\sgn \paren{k}$, where sgn is the sign function,
we may set the discrete derivative $\mc{S}_{h}^n$ and the discrete Hilbert transform $\mc{H}_h$ by
\begin{align*}
    \widehat{\paren{\mc{S}_{h}^n f}}_k          =&\paren{\inum k}^n\hat{f}_k, \mbox{ for } k=-\frac{N}{2}+1,\cdots, -1, 0, 1,\cdots, \frac{N}{2}-1,\\
    \widehat{\paren{\mc{S}_{h}^n f}}_\frac{N}{2}=&0,
\end{align*}
and
\begin{align*}
    \widehat{\paren{\mc{H}_h f}}_k          =&-\inum~\sgn \paren{k}\hat{f}_k, \mbox{ for } k=-\frac{N}{2}+1,\cdots, -1, 0, 1,\cdots, \frac{N}{2}-1,\\
    \widehat{\paren{\mc{H}_h f}}_\frac{N}{2}=&0,
\end{align*}
where we zero out the Fourier coefficient at $k=\frac{N}{2}$ for the asymmetry of the discrete Fourier Transform.
Also, when we apply the high derivative $\mc{S}_{h}^n$, the coefficients with the high frequencies $\hat{f}_k$ are amplified too much, so we set a filtering operator 
$\mc{P}$ from \cite{David-Michael-Keyang:Convergence-of-the-boundary-integral-method-for-interfacial-Stokes-flow} defined as 
\begin{align*}
    \widehat{\paren{\mc{P} f}}_k=\rho\paren{kh}\widehat{f}^k,
\end{align*}
where $\rho\in C^r\paren{[-\pi,\pi]}$ is a cutting function with $r>2$ satisfying
\begin{align*}
    \begin{split}
        \rho\paren{-x}=&\rho\paren{-x}\geq 0,\\
        \rho\paren{\pm \pi}=&\rho'\paren{\pm \pi}=0,\\
        \rho\paren{x}=&1 \mbox{ for } \abs{x}\leq\mu \pi, 0<\mu<1,
    \end{split}
\end{align*}
and we define the filtered derivative $\mc{D}_h^n=\paren{\mc{P}\mc{S}_{h}^n}$.
Moreover, we define the discrete anti-derivative $\mc{S}_{h}^{-1}$ by
\begin{align*}
    \widehat{\paren{\mc{S}_{h}^{-1} f}}_k=
    \begin{cases}
        \frac{1}{\inum k}\hat{f}_k, &\mbox{ for } k\neq 0, \frac{N}{2},\\
        &\\
        0, &\mbox{ for } k=0, \frac{N}{2}.
    \end{cases}
\end{align*}
We now approximate $\mc{R}^{\paren{-1}}$ by the trapezoidal rule.
Since as $\eta\rightarrow s$, $K^{\paren{-1}}(s,\eta)$ exist, we may define
\begin{align*}
    K_{kl}=&-\log \paren{\frac{\abs{\bm{X}_k-\bm{X}_l}}{2\abs{\sin \paren{ \frac{s_k-s_l}{2}}}}}I+\frac{\paren{\bm{X}_k-\bm{X}_l}\otimes\paren{\bm{X}_k-\bm{X}_l}}{\abs{\bm{X}_k-\bm{X}_l}^2},\\
    K_{kk}=&-\log\abs{\paren{\mc{S}_{h} \bm{X}}_k}I+\frac{\paren{\mc{S}_{h} \bm{X}}_k\otimes\paren{\mc{S}_{h} \bm{X}}_k}{\abs{\paren{\mc{S}_{h} \bm{X}}_k}^2}.
\end{align*}
Then, $\mathcal{R}^{\paren{-1}}\bm{Q}$ is approximated by $\mc{R}^{\paren{-1}}_h$,
\begin{align*}
    \paren{\mathcal{R}^{\paren{-1}}_h\bm{Q}}_k=-\frac{1}{4\pi}\mc{S}_{h}\paren{\sum_{l=0}^{N-1}K_{kl}\bm{Q}_l h}.
\end{align*}
Next, for the discretization of the commutator $[\mc{H},f]g$, one may directly use $\mc{H}_h \paren{fg}-f\mc{H}_h g$, but we choose another way to discretize it. 
First, we rewrite the commutator $[\mc{H},f]g$ as
\begin{align*}
    [\mc{H},f]g=\frac{1}{2\pi}\int_\mbs \cos \paren{\frac{s-\eta}{2}}\frac{f\paren{\eta}-f\paren{s}}{\sin \paren{\frac{s-\eta}{2}}} g\paren{\eta}d\eta.
\end{align*}
Since as $\eta \rightarrow s$,
\begin{align*}
    \frac{f\paren{\eta}-f\paren{s}}{\sin \paren{\frac{s-\eta}{2}}}\rightarrow -2 \PPD{s}{f}\paren{s},
\end{align*}
$[\mc{H},f]g$ is approximated by $[\mc{H}_h,f] g$,
\begin{align*}
    \paren{[\mc{H}_h,f] g}_k=-\frac{1}{2\pi}\paren{\sum_{l=0}^{N-1}H_{kl} g_l h},
\end{align*}
where
\begin{align*}
    H_{kl}=\cot \paren{\frac{s_k-s_l}{2}}\paren{f_l-f_k},\quad
    H_{kk}=-2 \paren{\mc{S}_{h} f}_k.
\end{align*}
Therefore, the semi-discretization is 
\begin{equation*}
\begin{split}
\p_t \varphi
=&\quad\frac{1}{4}\mc{H}_h \mc{S}_{h}^{3}\varphi+\frac{1}{4}\mc{H}_h\paren{\frac{1}{2}(\mc{P}_{h}\varphi+1)^3-\sigma\paren{\mc{P}_{h}\varphi+1}}\\
&-\frac{1}{4}\bm{n}\cdot[\mc{H}_h,\bm{n}]F_n-\frac{1}{4}\bm{n}\cdot[\mc{H}_h,\bm{\tau}]F_\tau
+\bm{n}\cdot\mc{S}_{h}\paren{\mc{R}^{\paren{-1}}_h\paren{F_n \bm{n}+F_\tau\bm{\tau}}}\\
:=&\quad\frac{1}{4}\mc{H}_h \mc{S}_{h}^{3}\varphi+\mc{N}_h\paren{\varphi,\sigma},
\end{split}
\end{equation*}
where
\begin{equation*}
\begin{aligned}
    F_n=\mc{P}_{h}^{3}\varphi +\frac{1}{2}(\mc{P}_{h}\varphi+1)^3-\sigma\paren{\mc{P}_{h}\varphi+1},\quad F_\tau=\mc{P}_{h}\sigma.
\end{aligned}
\end{equation*}
We solve $\sigma$ by
\begin{align*}
    \mc{L}_h\sigma=b_h\varphi,
\end{align*}
where
\begin{align*}
    \mc{L}_h\sigma=&\bm{\tau}\cdot\paren{-\frac{1}{4}\mc{H}_h+\mc{S}_{h}\mc{R}^{\paren{-1}}_h}\paren{\mc{S}_{h}\paren{\sigma \bm{\tau}}},\\
    b_h\varphi=&-\bm{\tau}\cdot\paren{-\frac{1}{4}\mc{H}_h+\mc{S}_{h}\mc{R}^{\paren{-1}}_h}\paren{\mc{S}_{h}{\paren{\paren{\mc{P}_{h}^{2}{\varphi}}\bm{n}-\frac{1}{2}\paren{\mc{P}_h \varphi-1}^2\bm{\tau}}}}.
\end{align*}
Then, the discrete position $\bm{X}_i$ is
\begin{align*}
    \bm{X}           =&\bm{X}_c+\mc{S}_{h}^{-1}\bm{\tau},\\
    \PPD{t}{\bm{X}_c}   =&\frac{1}{N}\paren{\sum_{l=0}^{N-1}\paren{\mc{R}^{\paren{-1}}_h\paren{F_n \bm{n}+F_\tau\bm{\tau}}}_l}.
\end{align*}
\subsection{Time Discretizations}
Now, we give the description of the time discretization. Set $t_m=m \Delta t$ and denote $f_j^{m}=f_j\paren{t_m}$, where $\Delta t$ is the time step.
We use an implicit-explicit Runge-Kutta type scheme, whose accuracy is second order.
The time scheme includes two steps.
The first step is computing $\varphi^{m+\frac{1}{2}}$.
We set the principle part $\frac{1}{4}\mc{H}_h \mc{S}_{h}^{3}\varphi$ in \eqref{eqNum}  as an implicit part $\frac{1}{4}\mc{H}_h \mc{S}_{h}^{3}\varphi^{m+\frac{1}{2}}$ and the nonlinear remaindering part as an explicit part $\mc{N}_h^m\paren{\varphi^m,\sigma^m}$.
The equations for $\varphi^{m+\frac{1}{2}}$ are
\begin{align*}
\begin{split}
    \frac{\varphi^{m+\frac{1}{2}}-\varphi^m}{\frac{1}{2}\Delta t}=&\quad\frac{1}{4}\mc{H}_h \mc{S}_{h}^{3}\varphi^{m+\frac{1}{2}}+\frac{1}{4}\mc{H}_h\paren{\frac{1}{2}(\mc{P}_{h}\varphi^m+1)^3-\sigma\paren{\mc{P}_{h}\varphi^m+1}}\\
&-\frac{1}{4}\bm{n}^m\cdot[\mc{H}_h,\bm{n}^m]F_n^m-\frac{1}{4}\bm{n}^m\cdot[\mc{H}_h,\bm{\tau}^m]F_\tau^m\\
&+\bm{n}^m\cdot\mc{S}_{h}\paren{\mc{R}^{\paren{-1},m}_h\paren{F_n^m \bm{n}^m+F_\tau^m\bm{\tau}^m}},
\end{split}
\end{align*}
and
\begin{equation*}
    \mc{L}_h^m\sigma^m=b_h^m\varphi^m.
\end{equation*}
The next step is computing $\varphi^{m+1}$ by a midterm method.
We set the principle part as $\frac{1}{4}\mc{H}_h \mc{S}_{h}^{3}\paren{\frac{\varphi^{m+1}+\varphi^m}{2}}$ and the remaindering part $\mc{N}_h^{m+1}\paren{\varphi^{m+1},\sigma^{m+1}}$. The equations of the second step are
\begin{align*}
\begin{split}
    \frac{\varphi^{m+1}-\varphi^m}{\Delta t}=
    &\quad\frac{1}{4}\mc{H}_h \mc{S}_{h}^{3}\paren{\frac{\varphi^{m+1}+\varphi^m}{2}}\\
    &+\frac{1}{4}\mc{H}_h\paren{\frac{1}{2}(\mc{P}_{h}\varphi^{m+\frac{1}{2}}+1)^3-\sigma\paren{\mc{P}_{h}\varphi^{m+\frac{1}{2}}+1}}\\
    &-\frac{1}{4}\bm{n}^{m+\frac{1}{2}}\cdot[\mc{H}_h,\bm{n}^m]F_n^{m+\frac{1}{2}}-\frac{1}{4}\bm{n}^{m+\frac{1}{2}}\cdot[\mc{H}_h,\bm{\tau}^{m+\frac{1}{2}}]F_\tau^{m+\frac{1}{2}}\\
    &+\bm{n}^{m+\frac{1}{2}}\cdot\mc{S}_{h}\paren{\mc{R}^{\paren{-1},m+\frac{1}{2}}_h\paren{F_n^{m+\frac{1}{2}} \bm{n}^{m+\frac{1}{2}}+F_\tau^{m+\frac{1}{2}}\bm{\tau}^{m+\frac{1}{2}}}},
\end{split}
\end{align*}
and
\begin{equation*}
    \mc{L}_h^{m+\frac{1}{2}}\sigma^{m+\frac{1}{2}}=b_h^{m+\frac{1}{2}}\varphi^{m+\frac{1}{2}}.
\end{equation*}
Then, we compute the position $\bm{X}^{m+1}$ with the following equations:
\begin{equation*}
\begin{split}
    \bm{X}^{m+1}           =&\bm{X}_c^{m+1}+\mc{S}_{h}^{-1}\bm{\tau}^{m+1},\\
    \frac{\bm{X}_c^{m+1}-\bm{X}_c^m}{\Delta t}  =&\frac{1}{N}\paren{\sum_{l=0}^{N-1}\paren{\mc{R}^{\paren{-1},m+\frac{1}{2}}_h\paren{F_n^{m+\frac{1}{2}} \bm{n}^{m+\frac{1}{2}}+F_\tau^{m+\frac{1}{2}}\bm{\tau}^{m+\frac{1}{2}}}}_l}.    
\end{split}
\end{equation*}

\subsection{Numerical Results}
To observe the behavior of the solutions computationally, we first take a star-shape initial condition with seven lobes:
\begin{equation*}
 \widetilde{\bm{X}}_0(\phi(s))=
    \begin{bmatrix}
    \paren{1+\frac{1}{4}\cos\paren{7\phi}}\cos \phi+ \frac{1}{8}\cos\paren{2\phi}\\
    \paren{1+\frac{1}{4}\cos\paren{7\phi}}\sin \phi+ \frac{1}{8}\sin\paren{2\phi}
    \end{bmatrix}.
\end{equation*}
where $\phi$ is the polar angular coordinate and is itself parametrized by the arclength $s$.
Compute the length $L$ of the closed curve, then define the initial condition as
\begin{equation*}
    \bm{X}_0\paren{s}:=\frac{2\pi}{L}\widetilde{\bm{X}}_0\paren{\phi\paren{\frac{L}{2\pi}s}}
\end{equation*}
The motion of $\Gamma_t$ can be found in Figure \ref{fig:ex01_1}.
We see how the contour evolves to a two-lobe steady state.
\begin{figure}[htbp]
    \centering
    \includegraphics[width=1\textwidth]{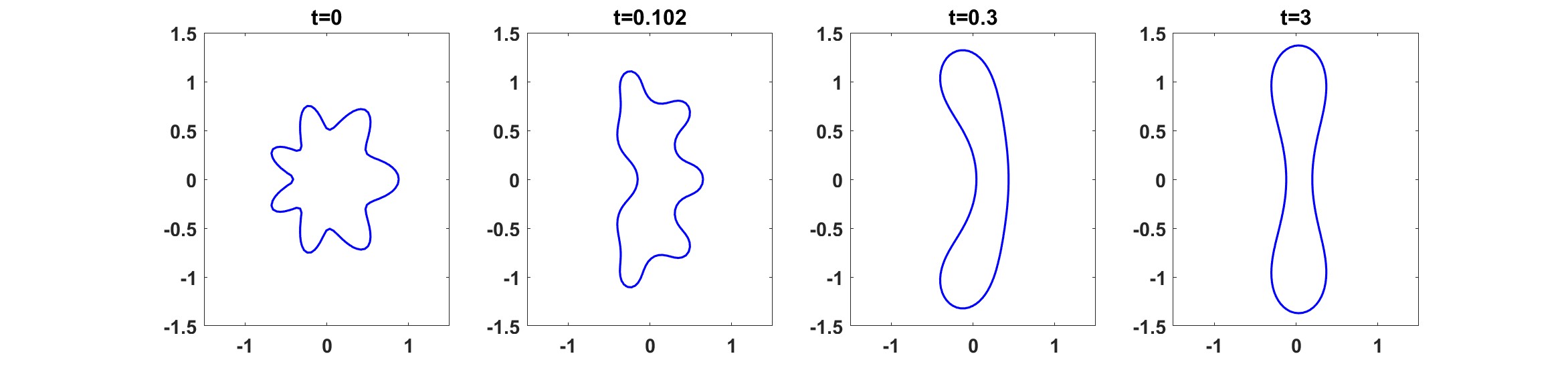}
    \caption{The motion of Example 1 with $N=128, \Delta t=$2e-4 }
    \label{fig:ex01_1}
\end{figure}

Next, we study convergence of the above case at $T=0.1$.
Denote the numerical results of $\varphi$ with the grid points $N$ at $s=kh$ as $\varphi_{N,k}$, and define relative errors $e_{N,i}$ as $k=0,1, \cdots, N-1$
\begin{equation*}
    e_{N,i}:=\varphi_{N,k}-\varphi_{2N,2k},
\end{equation*}
and their $\infty$-norm $\norm{\cdot}_\infty$ and $2$-norm $\norm{\cdot}_2$ as
\begin{equation*}
\begin{aligned}
    \norm{e_{N,i}}_\infty:=\sup_{i=0,\cdots,N-1}\abs{e_{N,i}},\\
    \norm{e_{N,i}}_2:=\sqrt{\frac{2\pi}{N}\sum_{i=0}^{N-1}e_{N,i}^2}.
\end{aligned}
\end{equation*}
We set three conditions:
(1) Fix the mesh size of $\Gamma_t$ to be $N=128$ and choose the time steps $\Delta t=2\times 10^{-5}, 10^{-5},5\times 10^{-6},2.5\times 10^{-6}$.
(2) Fix $\Delta t=5\times 10^{-6}$ and choose $N=64,128,256,512$.
(3) choose $N=64,128,256,512$ and set $\Delta t=k/N$ where $k=128\times10^{-5}$.
The results are shown in Figure \ref{fig:ex01 conv}.
We obtain second order convergence in all cases.

\begin{figure}[htbp]
    \centering
    \includegraphics[width=1\textwidth]{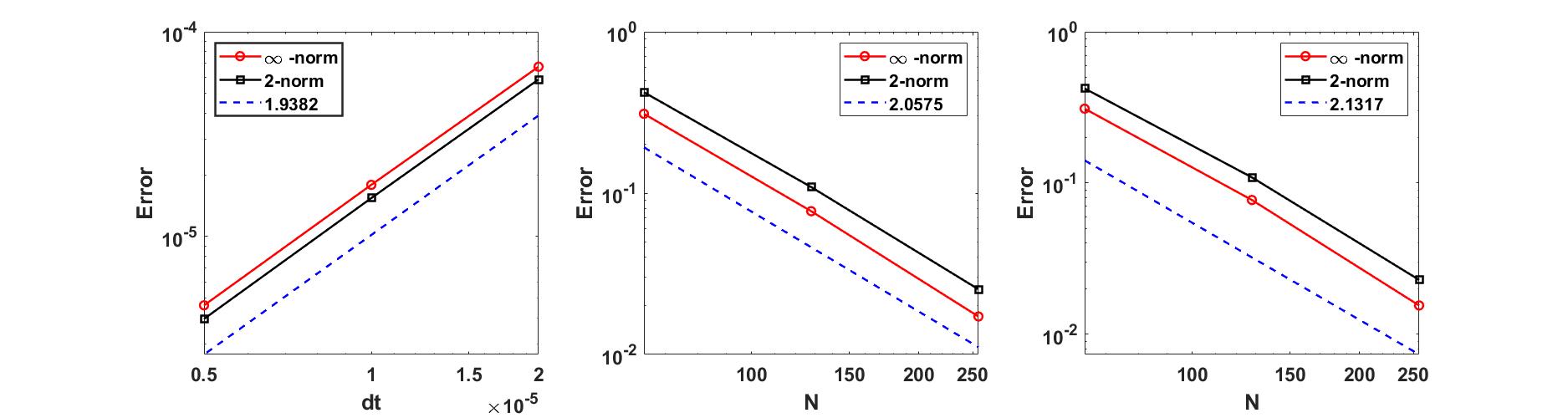}
    \caption{   These figures show the relative errors of Example 1 at $T=0.1$.
                Left:   we fix $N=128$ and decrease $\Delta t$.
                Middle: we fix $\Delta t=$5e-6 and increase $N$
                Right:  we increase $N$ and set $\Delta t=k/N$. 
                }
    \label{fig:ex01 conv}
\end{figure}

Next, we perform numerical simulations to observe asymptotic behavior.
In \cite{VeerapaneniRajBirosPurohit09}, equilibrium states with different number of lobes were computed.
We take the following near epicycloid curves with different lobes as initial conditions to observe their evolution:
\begin{equation*}
 \widetilde{\bm{X}}_0(\phi)=
    \begin{bmatrix}
    \cos \phi- \frac{1}{n+2}\cos\paren{(n+1)\phi}\\
    \sin \phi- \frac{1}{n+2}\sin\paren{(n+1)\phi}
    \end{bmatrix}.
\end{equation*}
In Figure $\ref{fig:ex02_1}$, we plot simulations for initial data $n=3$ for two different values of $N=128, 192$.
In both cases, $\Gamma_t$ converges quickly to a three-lobed shape, but eventually evolved into a two-lobed shape.
The face that one obtains different orientations for different values of $N$ indicates the strong sensitivity to initial data near
the three-lobe steady state, indicating that the three-lobe steady state is a saddle. This is also suggested by the fact that, 
when $N=192$, which is divisible by $3$, $\Gamma_t$ maintains a three-lobed shape for a longer time, until $t=7.5$.

Next, we observe three other cases, $n=4,5,6$.

The results are in Figure \ref{fig:ex03_1}.
In the first and the second rows, corresponding to the cases $n=4,5$, $\Gamma_t$ converges a two-lobed shape directly.
On the other hand, in the second row, $\Gamma_t$ deforms from a six-lobed shape into a three-lobed shape, and finally converges to a two-lobed shape.
This suggests that $m$-lobe steady states are generally saddles, and that they are connected to other steady states via heteroclinic orbits. 
In particular, if an $m_1$- and $m_2$-lobed steady state with $m_1>m_2\geq 3$ are connected by a heteroclinic orbit, it is expected that $m_1$ is divisible by $m_2$. 
\begin{figure}[htbp]
    \centering
    \includegraphics[width=1\textwidth]{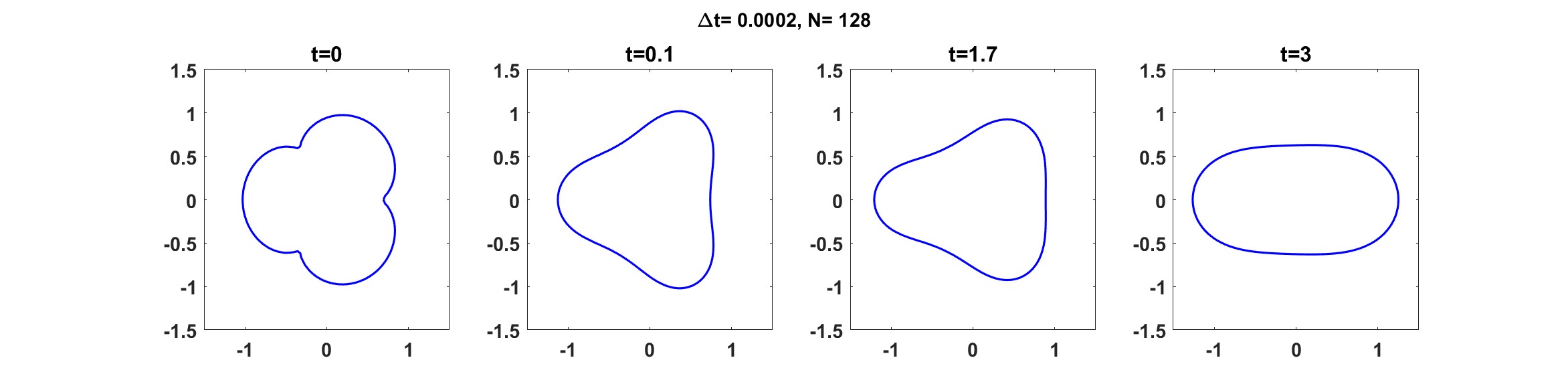}
    \includegraphics[width=1\textwidth]{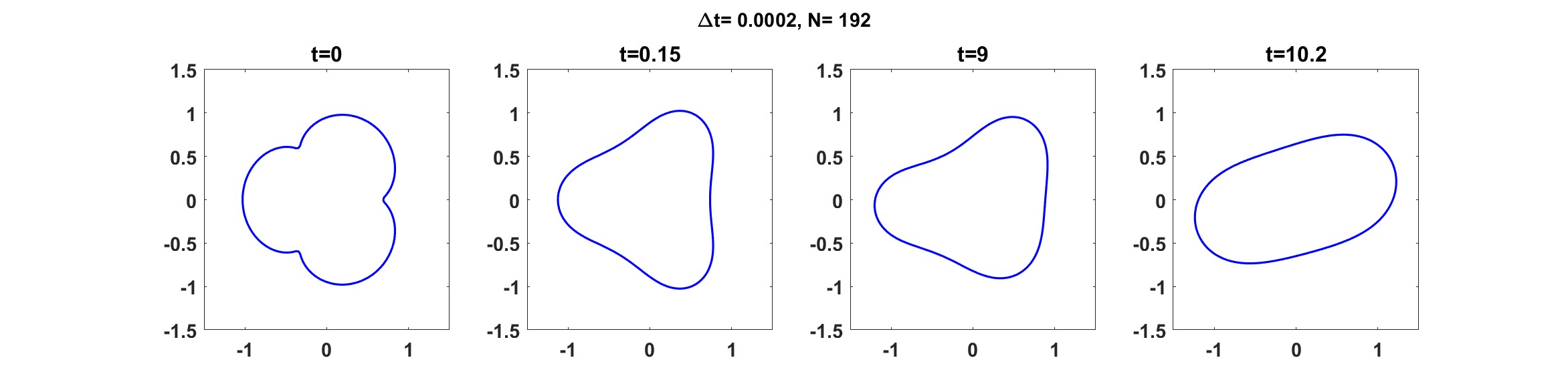}
    \caption{The motion of the near epicycloid initial condition with $3$ lobes.}
    \label{fig:ex02_1}
\end{figure}

\begin{figure}[htbp]
    \centering
    \includegraphics[width=1\textwidth]{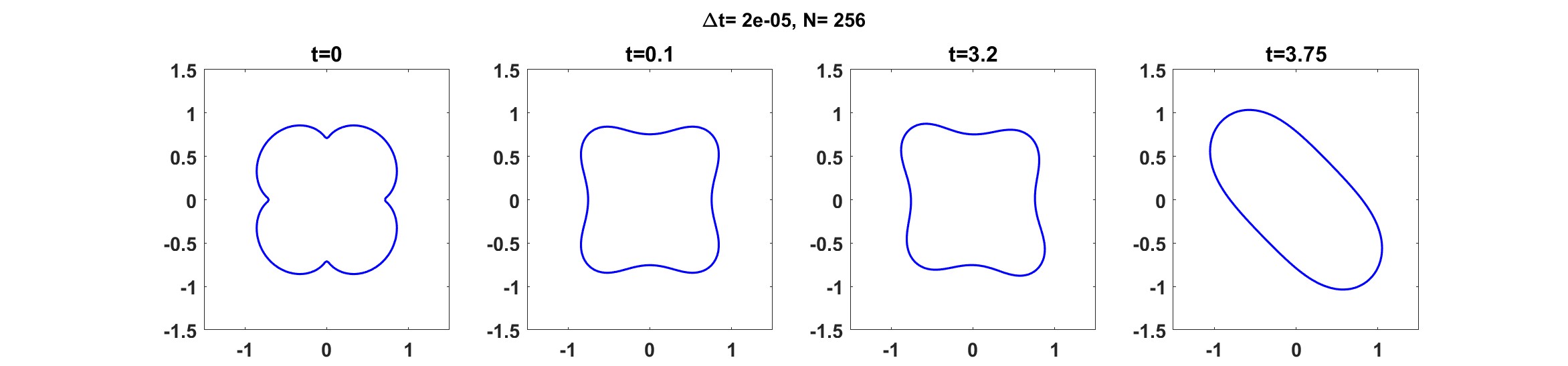}
    \includegraphics[width=1\textwidth]{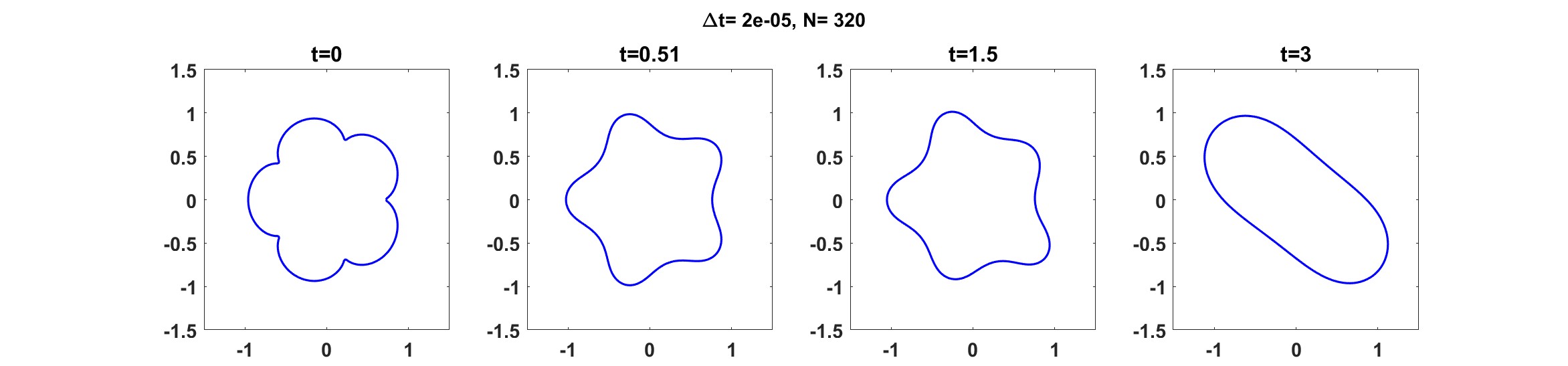}
    \includegraphics[width=1\textwidth]{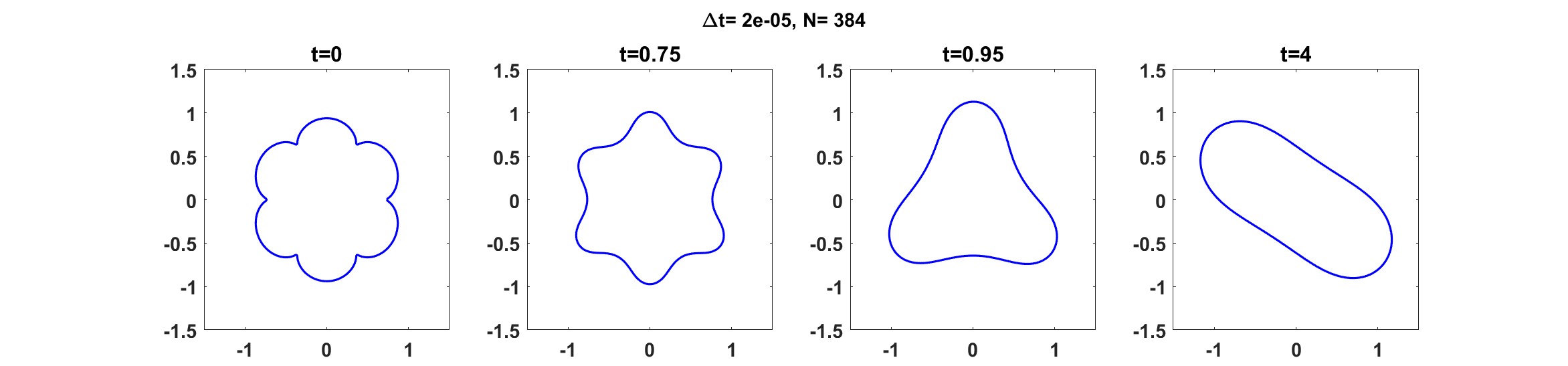}
     \caption{The motion of the near epicycloid initial conditions with $n$ lobes, $n=4,5,6$.}
    \label{fig:ex03_1}
\end{figure}

\section{Acknowledgements}

EGJ was supported by the European Union’s Horizon 2020 research and innovation programme under the Marie Skłodowska-Curie grant agreement CAMINFLOW No 101031111, and 
partially supported by the RYC2021-032877 research grant (Spain), the AEI projects PID2021-125021NA-I00, PID2020-114703GB-I00 and PID2022-140494NA-I00 (Spain), and the AGAUR project 2021-SGR-0087 (Spain). PCK was partially supported by NSF grant DMS-2042144 (USA) awarded to YM.
YM was partially supported by the NSF grant DMS-1907583, 2042144 (USA) and the Math+X award from the Simons Foundation.

\bibliographystyle{amsplain2link.bst}
\bibliography{references.bib}

\vspace{1cm}

\end{document}